\documentclass[12pt]{article}

\usepackage{soul}
\usepackage{hyperref}
\usepackage{authblk}
\usepackage[T1]{fontenc}
\usepackage{amsmath, amsthm, amsfonts}
\usepackage[initials]{amsrefs}
\usepackage{mathtools}
\usepackage[swedish, english]{babel}
\usepackage{amssymb}
\usepackage{csquotes}
\usepackage{mathrsfs}
\usepackage{bbm}
\usepackage[
top    = 2.25cm,
bottom = 2.25 cm,
left   = 2.25cm,
right  = 2.25cm]{geometry}
\usepackage{fancyhdr}
\usepackage{longtable}
\usepackage{multicol}
\usepackage{holtpolt}
\usepackage{graphicx}
\usepackage{caption}
\usepackage[font=small]{caption}
\usepackage{lipsum}
\usepackage{mwe}
\usepackage{subcaption}
\usepackage{xcolor}
\usepackage{tikz-cd}
\usepackage{booktabs}
\usepackage{lmodern}
\usepackage{microtype}
\usepackage{enumerate}
\usepackage{float}

\hypersetup{
    colorlinks,
    linkcolor={black},
    citecolor={blue!10!black},
    urlcolor={blue!10!black}
}

\newtheoremstyle{mytheoremstyle1} 
    {\topsep}                    
    {\topsep}                    
    {\itshape}                   
    {}                           
    {\bf \scshape}                   
    {}                          
    {.7em}                       
    {}  

\newtheoremstyle{mytheoremstyle2} 
    {\topsep}                    
    {\topsep}                    
    {}                   		
    {}                           
    {\scshape}                   
    {}                          
    {.7em}                       
    {}  

\theoremstyle{mytheoremstyle1}
\newtheorem{thm}{Theorem}[section]

\newtheorem{lem}[thm]{Lemma}
\newtheorem{prop}[thm]{Proposition}

\theoremstyle{mytheoremstyle2}
\newtheorem{hyp}[thm]{Hypothesis}

\theoremstyle{remark}
\newtheorem{rem}[thm]{Remark}

\def\Aa{\text{{\bf A}}}

\def\Bb{\text{{\bf B}}}
\def\C{\mathscr{C}}
\def\cC{\text{{\bf C}}}
\def\Cc{\overline{{\bf C}}}
\def\CC{\mathbb{C}}
\def\cchi{\underline{\chi}}

\def\H{\mathcal{H}}

\def\M{\mathbb{M}}

\def\F{\mathcal{F}}

\def\K{\mathcal{K}}
\def\N{\mathcal{N}}
\def\L{\mathcal{L}}
\def\Ll{{\bf L}}
\def\O{\mathcal{O}}

\def\Q{\mathcal{Q}}

\def\R{\mathbb{R}}
\def\Rr{{\bf R}}

\def\SS{\mathbb{S}}
\def\T{\mathcal{T}}
\def\TT{\tilde{\mathcal{T}}}
\def\U{\mathcal{U}}
\def\V{\mathcal{V}}

\def\ZZ{\mathbb{Z}}

\newcommand{\p}{\varphi}
\def\eeta{\tilde{\eta}}
\def\t{\tilde{t}}

\DeclareMathOperator{\Ker}{Ker}
\DeclareMathOperator{\ind}{ind}

\DeclareMathOperator{\Span}{span}
\DeclareMathOperator{\Id}{Id}
\DeclareMathOperator{\IIm}{Im}

\DeclareMathOperator{\D}{D}
\DeclareMathOperator{\sech}{sech}
\def\M{\mathcal{M}}
\def\Mm{\mathscr{M}}

\def \i{\text{i}}

\DeclareMathOperator{\uu}{u}

\newcommand\dif{\mathop{}\!\mathrm{d}}

\begin{document}

\title{Solitary waves in a Whitham equation \\
	with small surface tension}
\author{Mathew A. Johnson\thanks{Department of Mathematics, University of Kansas; \texttt{matjohn@ku.edu}}
\quad Tien Truong\thanks{Centre of Mathematical Sciences, Lund University; \texttt{tien.truong@math.lu.se}}
\quad Miles H. Wheeler\thanks{Department of Mathematical Sciences, University of Bath; \texttt{mw2319@bath.ac.uk}}
}
\maketitle
\normalsize

\abstract{Using a nonlocal version of the center manifold theorem and a normal form reduction, we prove the existence of small-amplitude generalized solitary-wave solutions and modulated solitary-wave solutions to the steady gravity-capillary Whitham equation with weak surface tension. 
Through the application of the center manifold theorem, the nonlocal equation for the solitary wave profiles 
is reduced to a four-dimensional system of ODEs inheriting reversibility.  Along particular parameter curves, relating directly to the classical gravity-capillary water wave problem, 
the associated linear operator is seen to undergo either a reversible $0^{2+}(\i k_0)$ bifurcation or a reversible $(\i s)^2$ bifurcation.  
Through a normal form transformation, the reduced system of ODEs along each relevant parameter curve is seen to be well approximated by a truncated system retaining only second-order or third-order
terms.  These truncated systems relate directly to systems obtained in the study of the full gravity-capillary water wave equation and, as such, the existence of generalized
and modulated solitary waves for the truncated systems is guaranteed by classical works, and they are readily seen to persist as solutions of the gravity-capillary Whitham equation
due to reversibility.  Consequently, this work illuminates further connections between the gravity-capillary Whitham equation and the full two-dimensional gravity-capillary 
water wave problem.
}


\section{Introduction}

%

In this paper, we consider the existence of small-amplitude solitary-wave solutions of the gravity-capillary Whitham equation
\begin{equation}\label{gra-cap-Whitham}
u_t + (\M_{g,d,T} u + u^2)_x = 0, 
\end{equation}
where here $\M_{g,d,T}$ is a Fourier-multiplier operator, acting on the spatial variable $x$, defined via its symbol
\[
m_{g,d, T}(\xi)= \left((g+ T \xi^2) \,  \frac{\tanh(d\xi)}{\xi}\right)^{1/2}.
\]
Here, $u(x,t)$ corresponds to the height of the fluid surface at position $x \in \R$ and time $t$, $g$ is the gravitational constant, $d$ is the undisturbed depth of the fluid, and $T\geq 0$ is the coefficient of surface tension. This symbol is precisely the phase speed for uni-directional waves in the full gravity-capillary water wave problem in \cite{RS_JohnsonBook,Whitham_book}.
In the absence of surface tension, that is when $T=0$, equation \eqref{gra-cap-Whitham} is referred to as {\it the gravity Whitham equation}, or
simply {\it the Whitham equation}, and was introduced 
by Whitham in \cite{Whitham_book,Whitham} as a full-dispersion generalization of the standard KdV equation.
In the case $T=0$, the bifurcation and dynamics of both periodic and solitary solutions of \eqref{gra-cap-Whitham} have been studied intensively over the last decade by many authors. It has been found that many high-frequency phenomena in water waves, such as breaking, peaking and the famous Benjamin-Feir instability, which do not manifest in the KdV or other shallow water,
are indeed manifested the Whitham equation. See for example \cite{BEP,EK,EJC,EGW,EW,HJ2,Hur} and references therein.

Given the success of the Whitham equation it is thus natural to consider the existence and dynamics of solutions when additional physical effects are incorporated.
In this work, we will concentrate on the existence of solitary-wave solutions\footnote{Note that the bifurcation and dynamics
of periodic waves has been previously studied in \cite{EJMR,HJ,JW}.} of \eqref{gra-cap-Whitham} with non-zero surface tension $T>0$.

It is straightforward to see that the properties of $m_{g,d,T}$ depend on the non-dimensional ratio
\begin{equation}\label{bond}
\tau = \frac{T}{gd^2},
\end{equation}
which is referred to as the Bond number.  In the full gravity-capillary water wave problem, it is known that the existence of solutions depends sensitively on whether
$\tau\in(0,1/3)$ or $\tau>1/3$, referred to as the weak- and strong-surface tension cases, respectively.  Indeed, in the case of strong surface tension the full gravity-capillary
water wave problem admits subcritical solitary waves of depression, i.e.~asymptotically constant traveling wave solutions with a unique critical point corresponding to an absolute minimum. See, for instance, \cite{AK2,AK}. Here, ``subcritical'' means the speed of the traveling wave is strictly less than the long-wave speed $m_{g,d,T}(0)=\sqrt{gd}$. If the traveling wave's speed is greater than $\sqrt{gd}$, it is said to be supercritical.  In the small surface tension case, however, considerably less is known about the existence of truly localized (e.g.~integrable) solitary waves.  It is known, however, that for small surface tension there exist generalized solitary waves, sometimes referred to as {\it solitary waves with ripples}. These correspond to bounded solutions of \eqref{steady-Whitham} which are (roughly) a superposition of a solitary wave and a co-propagating periodic wave with significantly smaller amplitude\footnote{We allow for the possibility of an asymptotic phase shift in the periodic wave between $x=-\infty$ and $x=+\infty$.}. See \cite{Sun,Baele1,SS,Lombardi1,Lombardi2}.  In particular, note that generalized solitary waves are not,
in fact, {\it solitary waves} in the traditional sense since they are not asymptotically constant at $x=\pm\infty$.  It is also known that there exist {\it modulated solitary waves}, which are bounded solutions of \eqref{steady-Whitham} with a solitary-wave envelope multiplying a complex exponential.  See, for instance,~\cite{IP, IK1990, BG}. Specifically, we note that~\cite{BG} proves the existence of geometrically distinct multipulse modulated solitary waves with exponential decay.  Illustrations of both generalized and modulated solitary waves can be seen in Figure \ref{GSW}.

\begin{figure}[t]
	\begin{center}
		\includegraphics{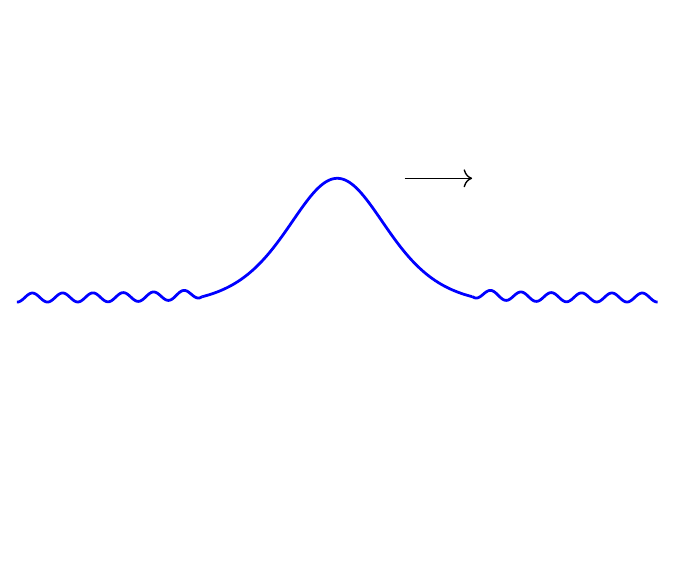} \hspace{0.3cm}
		\includegraphics{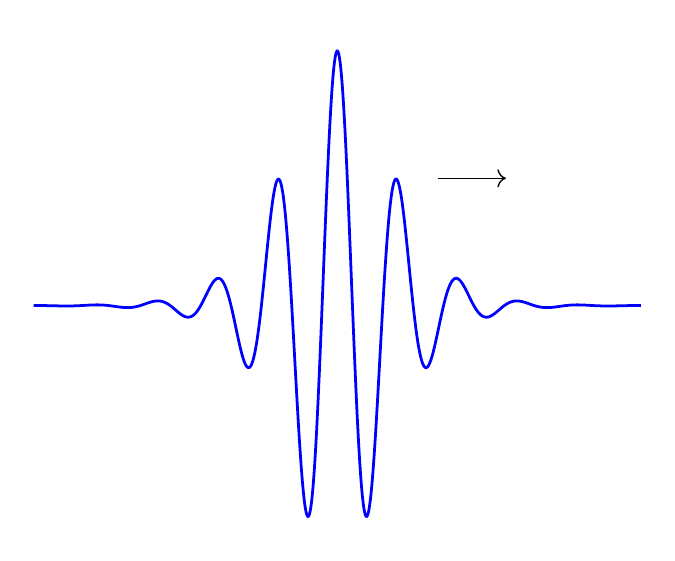}
		\captionsetup{width=.8\linewidth}
		\caption{Illustration of a generalized solitary-wave solution (left) and a modulated solitary-wave (of elevation) solution (right). The arrow indicates the direction of travel, that is, $c>0$.}
		\label{GSW}
	\end{center}
\end{figure}

Unfortunately, many of the existence proofs described above for the full gravity-capillary water wave problem 
rely fundamentally on classical dynamical systems techniques, requiring, in particular, that the equation governing 
the profile of the traveling wave be recast as a first-order system of ordinary differential equations.  Such techniques seem at first glance to not be  applicable to
the gravity-capillary Whitham equation \eqref{gra-cap-Whitham} due to the nonlocal operator $\mathcal{M}_{g,d,T}$.  However, \cite{FS,FS-corrig} recently derived a generalization
of the classical center-manifold theory that is applicable to a wide class of nonlocal problems, and this was further extended in \cite{TWW}  to an even wider class
of nonlocal problems which, as we will show, includes \eqref{gra-cap-Whitham}.  With this in mind, the primary goal of this paper is to use a nonlocal version
of the center manifold theorem and a corresponding normal form reduction to establish the existence of small amplitude generalized solitary and modulated solitary-wave solutions
to the gravity-capillary Whitham equation \eqref{gra-cap-Whitham} in the small surface tension case.  While such solutions were recently shown to exist 
in \cite{JW}, this work relies on direct implicit function theorem techniques.  Our goal is to attempt to establish similar results using a center-manifold reduction technique.

To begin our search for solitary waves, we note that a straightforward nondimensionaliz-ation converts \eqref{gra-cap-Whitham} to
\begin{equation}\label{gra-cap-Whitham2}
u_t+\left(\mathcal{M}_\tau u+u^2\right)_x=0
\end{equation}
where now $\mathcal{M}_\tau$  is a Fourier multiplier with symbol
\[
m_\tau(\xi) = \left((1+\tau \xi^2) \, \frac{\tanh(\xi)}{\xi} \right)^{1/2},
\]
and $\tau>0$ is the Bond number defined in \eqref{bond}.  Making the traveling wave ansatz $u(x,t) = \p(x - ct)$ in \eqref{gra-cap-Whitham2}
and  integrating yields the (nonlocal) stationary profile equation\footnote{Note that thanks to Galilean invariance, one can without loss of generality take the constant
of integration to be zero.  See Remark \ref{rem-supercritical-GSW} for more details.}
\begin{equation}\label{steady-Whitham}
\M_\tau \p - c\p + \p^2 = 0.
\end{equation} 
The profile equation \eqref{steady-Whitham} has received several treatments in recent years and theoretical frameworks for studying them are expanding. Existence results for~\eqref{steady-Whitham} include periodic waves by Hur \& Johnson~\cite{HJ} in 2015 and Ehrnstr\"om, Johnson, Maehlen \& Remonato~\cite{EJMR} in 2019, solitary (e.g.~ integrable) waves
for both strong and weak surface tension by Arnesen~\cite{A16} in 2016,  solitary waves of depression for strong surface tension $\tau > 1/3$ and subcritical wave speed $c<1$ by Johnson \& Wright~\cite{JW} in 2018, as well as generalized solitary waves for weak surface tension $\tau \in (0, 1/3)$ and supercritical wave speed $c>1$ also by~\cite{JW}.  
Each of these known results use either the implicit function theorem and a  Lyapunov-Schmidt reduction or appropriate variational methods.

In this work, we utilize instead an approach based on the recent nonlocal center manifold reduction technique developed by Faye \& Scheel \cite{FS,FS-corrig}
and further refined by Truong, Wahl\'en \& Wheeler \cite{TWW}.
As we will see, this set of techniques provides a unified approach for proving existence of both periodic and solitary waves for~\eqref{steady-Whitham}. The nonlocal center manifold theorem 
bears resemblance to its classical local counterpart, that there exists a neighborhood in a uniform locally Sobolev space where the nonlocal equation is equivalent 
to a local finite-dimensional system of ODEs. After this reduction, tools for ODEs can be applied to find an approximate solution and then to investigate 
its persistence. Provided the solution is sufficiently small in the uniform local Sobolev norm it qualifies as a true solution to the original nonlocal 
equation.  
So, previously mentioned small-amplitude waves are likely to be included in the center manifold. 
However, this framework does not fit nonlocal equations from hydrodynamics. To remedy this, 
Truong, Wahl\'en \& Wheeler~\cite{TWW} have extended this result to a larger class of nonlocal equations. They also demonstrate the strength of this 
reduction technique and exemplify how to extract qualitative information on the solutions from the reduced ODE, which they use to construct an extreme 
solitary wave for the gravity Whitham equation. This reduction technique is also available for local quasilinear problems \cite{CWW2}. 

\begin{rem}\label{rem-ncm_refined}
As noted above, the nonlocal center manifold theorem developed in \cite{FS,FS-corrig} does not directly apply to the profile equation \eqref{steady-Whitham}.
Indeed, one of the hypotheses of Faye \& Scheel's result is that both of the functions\footnote{Here and throughout, $\F$ denotes the Fourier transform.  For the specific
normalization used here, see equation \eqref{FT} below.}
 $\F^{-1}(m_\tau^{-1})$ and $\partial_x\F^{-1}(m_\tau^{-1})$ are integrable  and exhibit exponential decay, which are highly non-trivial properties.
While the exponential decay and integrability of $\F^{-1}(m^{-1}_\tau)$ was recently established in \cite{EJMR}, this reference
also unfortunately shows that $\partial_x\F^{-1}( m_\tau^{-1})$ is not an integrable function.  By carefully considering the methodologies used in \cite{FS,FS-corrig},
Truong, Wahl\'en and Wheeler were recently able to circumvent this difficulty in \cite{TWW}, where they present a refinement of the result in \cite{FS,FS-corrig} which does not rely on the integrability of $\F^{-1}( m_\tau^{-1})'$. The authors explain the purpose of this hypothesis is to establish the Fredholmness
of the linearized operator obtained by linearizing \eqref{steady-Whitham} about $\p=0$. Fortunately, this can be checked directly by other means
for equations of the form \eqref{steady-Whitham}, which is one of the achievements of the refinement \cite{TWW}.  It is technically this refinement which we use in our analysis. For comparison, we note that such properties of $\F^{-1}(m_\tau^{-1})$ are not needed in the implicit function theorem approach used 
by Johnson \& Wright \cite{JW}. 
\end{rem}

We now provide an outline of the paper, as well as state the main results.
We begin in Section \ref{sec-lin_op}  by inverting the operator $\mathcal{M}_\tau$ in \eqref{steady-Whitham}, thereby recasting the profile equation
into the form studied in \cite{FS,FS-corrig,TWW}.  We then study the equation
\[
m_\tau(\xi)-c = \left((1 + \tau \xi^2) \frac{\tanh(\xi)}{\xi}\right)^{1/2}-c = 0,
\] 
which gives solutions to the linearized equation of~\eqref{steady-Whitham} about the trivial solution $\p=0$.
By writing $c^{-2} = \alpha$, $\tau c^{-2} = \beta$ and rearranging the terms, the above equation is recognized as the well-known linear dispersion 
relation for purely imaginary eigenvalues in two-dimensional capillary-gravity water wave equations, modeling the motion of a perfect unit-density 
fluid with irrotational flow under the influence of gravity and surface tension in finite depth: see, for example, the works of Kirchg\"assner~\cite{Kirchgassner},
Buffoni, Groves \& Toland~\cite{BGT}, Amick \& Kirchg\"assner~\cite{AK} and Diass \& Iooss~\cite{DI}. These classical bifurcation 
curves in the $(\beta, \alpha)$-plane naturally guide us in selecting two parameter curves where, restricting ourselves now to the case of small surface tension $\tau\in(0,1/3)$,
we expect generalized solitary-wave and modulated  solitary-wave solutions could be found: see Figure~\ref{bif-diag} below.  We further establish a key Fredholm property
for the associated linearized operator in Section \ref{sect-lin-op} which is required for the application of the center manifold theorem.

With this preliminary linear analysis completed, we then turn towards applying the nonlocal center manifold theorems from \cite{FS,FS-corrig,TWW} to the profile
equation \eqref{steady-Whitham}.  These results, as mentioned before, reduce the nonlocal profile equation considered here to a local ODE
near the equilibrium and provide an algorithmic method of approximating the local ODE. We often refer to this local ODE as the reduced ODE. For completeness, we state the general center manifold theorem from \cite{TWW} in Appendix \ref{CMT-appendix}, and we apply it in Section \ref{sect-CMT} along both bifurcation curves in the $(\beta, \alpha)$-plane of interest.  

In Section \ref{sect-reduced-r1} we approximate the reduced ODE near the bifurcation curve where generalized solitary-wave solutions are expected to be found as a result of an $0^{2+}(\i k_0)$ reversible bifurcation.  
%
%
Up to a standard normal form reduction, rescaling and truncating the nonlinearity, this ends up being almost identical to the normal 
form equation obtained by Iooss \& Kirchg\"{a}ssner in \cite{IK} in their analysis 
of the full gravity-capillary water wave equations.  In particular, the truncated reduced ODE in this case admits an explicit family of small amplitude generalized solitary-wave solutions which are then 
shown to persist as solutions of \eqref{steady-Whitham} by a reversibility argument.  Putting this all together establishes our first main result.

\begin{thm}[Existence of Generalized Solitary Waves]
	\label{thm-existence-GSW-intro} For each sufficiently small $\mu \in \R$, there exists a family of generalized solitary waves to the gravity-capillary Whitham equation with wave speed $c= 1+\mu$ and $\tau < 1/3$, 
	given by
	\[\begin{split}\p(x)&=\frac{3}{2}|\mu|\rho^{1/2} \mathrm{sech}^2\left(\frac{\rho^{1/4}\sigma^{1/2}|\mu|^{1/2}x}{\sqrt{2}}\right) + \frac{\mu}{2}(1-\mathrm{sgn}(\mu)\rho^{1/2}) \\
	&\quad +|\mu|k^{1/2} \cos\Big( (k_0+\O(\mu))x + \Theta_* + \O(\mu)\Big) + \O(\mu^2\rho^{1/2}),
	\end{split}\]
  where $\Theta_* \in \R / 2\pi \ZZ$ is arbitrary, $\sigma = (1/3-\tau)^{-1}$, $\rho = 1+24k$, $k_0>0$ is such that $m_\tau(k_0)=1$, and 
  $k =\O(|\mu|^{-1-2\kappa})$ for any $\kappa \in [0, 1/2)$.
\end{thm}

It is interesting to note that the above allows for an 
asymptotic phase shift in the cosinus term between $x = -\infty$ and $x=\infty$ of order $\O(\rho^{1/4}|\mu|^{1/2})$.  
Further, after a Galilean change of variables, all the generalized solitary-wave solutions found above may be seen to have supercritical 
wave speed $c>1$: see Remark~\ref{rem-supercritical-GSW} below for details. Note, however, this result does not establish that some waves have asymptotic oscillations which are exponentially small in relation to the solitary term as in Johnson \& Wright~\cite{JW}. On the other hand, we are able to allow for a more general asymptotic phase shift between $x=\pm\infty$.

In Section~\ref{sect-reduced-r2} we analogously treat the bifurcation curve in the $(\beta, \alpha)$-plane where modulated solitary waves are expected to be found as a result of a Hamiltonian-Hopf bifurcation, also known as an $(\i s)^2$ bifurcation. By computing the necessary center manifold coefficients and performing the appropriate normal form reduction, we again find the results
from \cite{IP} applicable, thus establishing our second main result. 

%

\begin{thm}[Existence of Modulated Solitary Waves]
	\label{thm-existence-r2-intro}  Fix  $s>0$ and set 
\begin{align*}
c_0^2 = \left(\frac{s^2}{2\sinh^2(s)}+ \frac{s}{2\tanh(s)}\right)^{-1}, \quad \tau_0 = c_0^2 \left(-\frac{1}{2\sinh^2(s)}+\frac{1}{2s\tanh(s)}\right),
\end{align*}
  so that $c_0^{-2}(1+\tau_0)\sinh(s) = s \cosh(s)$.  Then, for  $\mu<0$ sufficiently small, there exist two distinct modulated solitary-wave solutions to the gravity-capillary Whitham equation with amplitude of order $\O(|\mu|^{1/2})$, surface tension $\tau_0$
and subcritical wave speed $c_0 + \mu <1$. More precisely, the modulated solitary-wave solutions are described asymptotically via
	\[\p(x)=\sqrt{\frac{-8 q_0 \mu}{q_1}} \sech(\sqrt{q_0\mu} \, x) \cos\left(sx + \Theta_* + \O\big(|\mu|^{1/2}\big)\right) + \O(\mu^2), \hspace{0.2cm} \Theta_* \in \{0, \pi\},\]  
	 which have an asymptotic phase shift between $x=-\infty$ and $x=\infty$ of order $\O(|\mu|^{1/2})$.  Here, the coefficients $q_0$ and $q_1$ are
	\[\begin{split} 
	q_0 = -\frac{2}{m_\tau''(s)}, \quad q_1 = -\frac{4(-c_0+m_\tau(2s))^{-1} + 8(1-c_0)^{-1}}{m''_\tau(s)},
	\end{split}\] 
	and are both negative.
\end{thm}

The solutions constructed in Theorem \ref{thm-existence-r2-intro} correspond to (distinct) modulated solitary waves of elevation ($\Theta_*=0$) and depression
($\Theta_*=\pi$).  Note that 
Figure~\ref{GSW} depicts a modulated solitary-wave solution of elevation.

\begin{rem}
As mentioned above, the center manifold methodology used here provides a unified approach for proving existence of both periodic and solitary waves for~\eqref{steady-Whitham}.
Consequently, one could continue the above line of investigation to establish the existence of other classes of solutions as well including, for instance,
the subcritical, small amplitude solitary waves of depression in the case of large surface tension $\tau>1/3$ constructed in \cite{JW}. This specific case is very similar to the gravity Whitham equation studied by Truong, Wahl\'{e}n and Wheeler~\cite{TWW} and is thus excluded here. 
\end{rem}

This paper provides further connections between the model equation~\eqref{gra-cap-Whitham} and the two-dimensional gravity-capillary water wave problem. It also exemplifies the application of nonlocal center manifold reduction in existence theory. A natural continuation of this paper could be to 
investigate the existence of multipulse modulated solitary-wave solutions as in~\cite{BG}, as well as bifurcation phenomena in other parameter regions.

\subsubsection*{Notation}
The following notation will be used throughout this work.
\begin{itemize}
        \item[--] For $\sigma \in \R$, we define the {\it $\sigma$-weighted $L^p$ spaces}
          \[ L^p_{\sigma}:=\left\{f\colon \R \to \R  \, \Big| \, \int_\R |f|^p \omega_\sigma^p  \dif x \right\}. \]
        Here the {\it weight function} $\omega_\sigma\colon \R \to \R$ is positive and smooth. Also, $\omega_\sigma$ is constantly 1 on $[-1,1]$ and equals $\exp(\sigma |x|)$ for $|x| \ge 2$. 
 	\item[--] Similarly, we define the {\it weighted Sobolev spaces}       
        \[W^{m,p}_\sigma := \left\{f\colon \R \to \R \, \Big| \, f^{(n)} \in L^p_\sigma, \, \, \text{for} \, \, 0 \leq n \leq m\right\}. \]
        We have the natural inclusions $W_{\sigma_2}^{m,p} \subset W_{\sigma_1}^{m,p}$ whenever $\sigma_1 < \sigma_2$. For $p=2$, we denote the Hilbert space $W^{m,2}_\sigma$ by $H^m_\sigma$. 
	\item [--] The non-weighted Sobolev spaces are denoted by $W^{m, p}$ and the special case $W^{m,2}$ is denoted by $H^m$. 
  \item[--] The uniform locally $H^m$ space is
	\[H^m_{\uu}:=\left\{ f \colon \R \to \R \,  \Big| \, \|f\|_{H^m_{\uu}} < \infty\right\} \hspace{0.2cm}\text{with}\hspace{0.2cm} \|f\|_{H^m_{\uu}}:=\sup_{y \in \R}\|f(\, \cdot \, + y)\|_{H^m([0,1])}.\]

	\item[--] We use the following scaling of the Fourier transform:
\begin{equation}\label{FT}
\F f(\xi) = \hat{f}(\xi) \coloneqq \int_\R f(x)\exp(-\i x\xi)\dif x \hspace{0.5cm}\text{and}\hspace{0.5cm}\F^{-1}g(x) =  \frac{1}{2\pi} \F g(-x).
\end{equation}
%
\end{itemize}

\section{The operator equation}\label{sec-lin_op}

In this section, we begin our study of the nonlocal profile equation \eqref{steady-Whitham}.  Observe that since $m_\tau$ is strictly positive on $\R$,
the operator $\M_\tau$ is invertible on any Fourier based space.  We denote the inverse of $\M_\tau$ by $\L_\tau$, defined via
\[
\widehat{\,\L_\tau f\,}(\xi) = \ell_\tau(\xi)\hat{f}(\xi),\quad \ell_\tau(\xi)\coloneqq m_\tau(\xi)^{-1}.
\]
In particular, the profile equation \eqref{steady-Whitham} can be written in the ``smoothing'' form
\begin{equation}\label{smoothing}
\p - cK_\tau * \p + K_\tau * \p^2 = 0,
\end{equation}
where here $K_\tau:=\F^{-1}\ell_\tau$ denotes the convolution kernel corresponding to the operator $\L_\tau$. 
Observe that \eqref{smoothing} is similar to the profile equation for the gravity Whitham equation (i.e. \eqref{gra-cap-Whitham} with $T=0$), but
now  with a \emph{nonlocal} nonlinearity.

As we seek small amplitude solutions of \eqref{smoothing}, we begin by linearizing \eqref{smoothing} about $\p = 0$ which, after applying the Fourier transform, yields the equation
\[
(1 - c\ell_\tau(\xi)) \widehat{v}(\xi) = 0,
\]
which we seek to solve for non-trivial $v\in L^2(\R)$.  This motivates considering the equation 
\begin{equation}\label{lin-dis-rel} 
1 - c\ell_\tau(\xi) = 0,\quad{\rm i.e.}\quad \xi \cosh(\xi) = \left(\frac{1}{c^2} + \frac{\tau}{c^2}\xi^2\right) \sinh(\xi).
\end{equation}
By setting 
\[
\alpha = \frac{1}{c^2} \hspace{0.5cm}\text{and}\hspace{0.5cm}\beta = \frac{\tau}{c^2},
\]
equation~\eqref{lin-dis-rel} is recognized as the well-known linear dispersion relation for purely imaginary eigenvalues in the two-dimensional water wave equations in finite depth:
see, for example, \cite{RS_JohnsonBook,Whitham_book}.
The importance of~\eqref{lin-dis-rel} in finding solutions to the gravity-capillary water wave equations was recognized by Kirchg\"assner~\cite{Kirchgassner}, followed by a multitude of other papers (see e.g.~\cite{AK, DI, IK, BGT}). Looking at the bifurcation curves in the $(\beta, \alpha)$-planes for the classical gravity-capillary water wave problem, it is natural to expect the following:
\begin{itemize}
	\item that modulated solitary-wave solutions may be found as a result of a Hamiltonian--Hopf bifurcation, when crossing the curve
	\[
	C_2 = \bigg\{(\beta, \alpha) = \bigg(-\frac{1}{2\sinh^2(s)} + \frac{1}{2 s \tanh(s)}, \frac{s^2}{2\sinh^2(s)} + \frac{s}{2\tanh(s)}\bigg) \, \Big|\, s \in [0, \infty)\bigg\}
	\]
	from below;
	\item that generalized solitary-wave solutions may be found as a result of an $0^{2+}(\i k_0)$ bifurcation, when crossing the curve 
	\[
	C_3 = \Big\{(\beta, \alpha) \, \Big|\, \beta \leq \frac{1}{3} \, \, \text{and} \, \, \alpha = 1 \Big\}
	\]
	either from above or below. Here, $k_0 \in \R$ satisfies equation~\eqref{lin-dis-rel} for a fixed $\beta$ along $C_3$;
	\item and that solitary-wave solutions of depression may be found as a result of an $0^{2+}$ bifurcation, when crossing the curve
	\[
	C_4 = \Big\{(\beta, \alpha) \,\Big|\, \beta \geq \frac{1}{3} \, \, \text{and}\,\, \alpha = 1 \Big\}
	\]
	from above.
\end{itemize}

There is an additional curve $C_1$ in the $(\beta,\alpha)$-plane along which one may expect the existence of multi-pulse solitary waves~\cite{BGT}. The argument in~\cite{BGT} uses the Hamiltonian structure of the full water wave problem. While equation~\eqref{gra-cap-Whitham} exhibits a variational formulation in the form investigated by Bakker \& Scheel~\cite{BS}, its smoothing form~\eqref{eq} below does not. As $m_\tau$ does not have an $L^1$ Fourier transform, using results in~\cite{BS} would therefore call for a careful examination and adaptation. Thus, it is more appropriate to consider this bifurcation phenomenon in a separate paper and we will not comment further regarding $C_1$. See Figure~\ref{bif-diag} for depictions of the curves $C_1,C_2,C_3$ and $C_4$ in the $(\beta, \alpha)$-plane.  

It is illustrative to understand what these curves mean in terms
of the physical $(\tau,c)$ parameters.  For example, crossing $(\beta_0, \alpha_0)=(\beta_0,1)\in C_3$ 
corresponds to studying \eqref{smoothing} for $\beta=\beta_0$ and $\alpha=\alpha_0-\tilde{\mu}=1-\tilde{\mu}$ for some $|\tilde{\mu}| \ll 1$ which, in terms of $\tau$ and $c$,
gives
\begin{equation}\label{parameter-r1}
\tau=\beta_0 c_0^2<\frac{1}{3}~~~{\rm and}~~~ c=c_0+\mu~~{\rm with}~~c_0=1~~{\rm and}~~|\mu|\ll 1.
\end{equation}
Thus, crossing the curve $C_3$ correspond to waves with weak surface tension, while the wave speed is nearly critical.  Likewise,
crossing a point $(\beta_0,\alpha_0)\in C_2$ from below corresponds to studying \eqref{smoothing} with $\beta=\beta_0$
and $\alpha=\alpha_0+\tilde{\mu}$ with $0<\tilde\mu\ll 1$ which, in terms of $\tau$ and $c$ gives
\begin{equation}\label{parameter-r2}
\tau=\beta_0 c_0^2~~{\rm and}~~c=c_0+\mu,~~{\rm with}~~ c_0=\alpha_0^{-1/2}~~{\rm and}~~0<-\mu\ll 1.
\end{equation}
Since $\alpha_0>1$, it follows that $\tau\in(0,1/3)$, i.e.~the surface tension is again weak, and the speed $c$ is subcritical.  By similar reasoning, crossing
the curve $C_4$ from above corresponds to strong surface tension, i.e.~$\tau>1/3$, and subcritical speeds.  Note that in this work, 
we focus only on bifurcation phenomena connected to the curves $C_2$ and $C_3$.  
The $0^{2+}$ bifurcation along $C_4$, as one might expect, resembles the one already covered in~\cite{TWW} for the gravity Whitham equation 
and is thus excluded here. 

\begin{figure}[t]
	\begin{center}
		\includegraphics[scale=1.2]{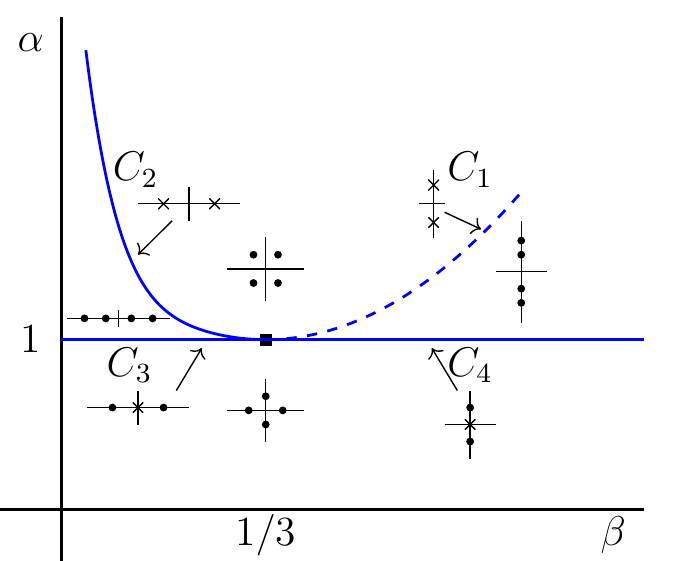}
		\captionsetup{width=.8\linewidth}
		\caption{Sketch of the bifurcation curves $C_1,C_2, C_3$ and $C_4$ along with zeros of the function $1-c\ell_\tau(\xi)$, which is the same as those of $(\alpha + \beta \xi^2)\sinh(\xi)-\xi\cosh(\xi)$. Here, dots and crosses represent algebraic multiplicity one and two, respectively.}
		\label{bif-diag}
	\end{center}
\end{figure}
%
%

As mentioned previously, we approach the above bifurcation phenomena for the nonlocal profile equation \eqref{smoothing} by following
the center manifold reduction strategy in \cite{FS,FS-corrig,TWW}.  To this end, we rewrite \eqref{smoothing} as
\begin{equation}\label{eq}
\T \p + \N(\p, \mu) = 0,
\end{equation}
which is now of the structural form studied in Faye \& Scheel~\cite{FS,FS-corrig}, where  here
\[
\T \colon \p \mapsto \p -  c_0K_\tau * \p \hspace{0.5cm} \text{and} \hspace{0.5cm} \N \colon (\p, \mu) \mapsto K_\tau * (\p^2 - \mu \p).
\]
Since $\tau$ is fixed, the subscript $\tau$ will be dropped for notational convenience. Our goal is to study the operator equation \eqref{eq} for parameters $\tau,c_0$ corresponding to \eqref{parameter-r1}, corresponding to crossing $C_3$, as well as for $\tau,c_0$ satisfying \eqref{parameter-r2} corresponding to crossing $C_2$. The first step of this analysis is to understand the linear operator $\T$, which we now turn to studying.

\section{\texorpdfstring {The linear operator $\T$}{The linear operator T}}
\label{sect-lin-op}

Fix $\tau\in(0,1/3)$.
As preparation for our forthcoming bifurcation analysis, and following the general strategy in \cite{FS,FS-corrig,TWW}, in this section we study
the linear operator\footnote{Recall that since $\tau$ is fixed, for convenience the corresponding subscript will be dropped from $K_\tau$ and $\ell_\tau$.}
\[
\T\colon \p \mapsto \p - c_0 K * \p , \hspace{0.5cm} H^5_{-\eta} \to H^5_{-\eta},
\]
for  $\eta > 0$ along the two parameter curves~\eqref{parameter-r1} and~\eqref{parameter-r2}.  Note that $\T$ is precisely the linearization
of \eqref{eq} about the trivial solution $\p=0$ and, as such, it is crucial to understand the Fredholm and invertibility properties of $\T$ along
the curves $C_2$ and $C_3$. Key to this analysis is an understanding of convolution kernel $K$.  The relevant properties are detailed in the following result.

\begin{prop} \label{K-prop} The convolution kernel $K$ is even. Moreover, we have
\begin{enumerate}[(i)]
\item the singularity of $K$ as $|x|\to 0$ is 
\[\lim_{x\to 0} \sqrt{|x|}K(x) = \frac{1}{\sqrt{2\pi \tau}};\]
\item $K$ has exponential decay as $|x|\to \infty$, that is
\[|K(x)| \lesssim \exp(-\eta |x|) \hspace{0.5cm} \text{for} \hspace{0.5cm} |x| > 1,\]
where $0 < \eta < \eta^* \coloneqq \min\{1/\sqrt{\tau}, \pi/2\}$. 
\end{enumerate}
\end{prop}

For a proof, see Ehrnstr\"om, Johnson, Maehlen \& Remonato~\cite[Theorem 2.7]{EJMR}. An immediate consequence is that $K \in L^1_{\eta}$ for $\eta \in (0, \eta^*)$ and hence, by a straightforward application of Young's inequality, that for such $\eta$ the linear operator 
\[
\T \colon H^5_{-\eta} \to H^5_{-\eta}
\]
is bounded regardless of the choice of $c_0$. A proof of this claim is found in~\cite{TWW} but is repeated here for the readers' convenience. We estimate
\[\begin{split}
\big\|K*\p\big\|^2_{L^2_{-\eta}} &\lesssim \int_\R \left(\int_\R K(y) \p(x-y) \dif y\right)^2 \exp(-2\eta|x|)\dif x \\
&\leq \int_\R \left(\int_\R K(y)\p(x-y)\dif y\right)^2 \exp(-2\eta|x-y|+2\eta|y|)\dif x \\
& = \int_\R \left(\int_\R K(y) \exp(\eta|y|) \cdot \p(x-y)\exp(-\eta|x-y|) \dif y\right)^2 \dif x \\
& = \Big\|\big(K\cdot \exp(\eta|\,\cdot \,|)\big) * \big(\p\cdot \exp(-\eta|\, \cdot \,)\big)\Big\|_{L^2}^2 \\
& \leq \|K\|^2_{L^1_{\eta}} \cdot \|\p\|_{L^2_{-\eta}}^2.
\end{split}\]
This establishes that $\T$ is bounded on $L^2_{-\eta}$. Using that 
\[\frac{\dif}{\dif x^n} (K*\p) = K*\left(\frac{\dif}{\dif x^n}\p\right), \quad \text{for all $n \ge 0$,}\] 
the boundedness of $\T$ on $H^5_{-\eta}$ readily follows. Note that the works \cite{FS,FS-corrig} additionally require $K'\in L^1_{\eta}$ which, by Proposition \ref{K-prop}, 
does not hold in this case.  Next, we follow Truong, Wahl\'{e}n \& Wheeler~\cite{TWW} and study the Fredholm properties of $\T$ using theory for pseudodifferential operators in non-weighted Sobolev spaces $H^5$ from Grushin~\cite{Grushin} (see also Appendix A). 

To this end, we fix $\eta\in(0,\eta^*)$ and consider the conjugated operator
\[
\tilde{\T} \coloneqq M^{-1} \circ \T \circ M, \hspace{0.5cm} H^5 \to H^5,
\]
where $M:H^5\to H^5_{-\eta}$ is multiplication with the strictly positive function $\cosh(\eta \, \cdot \,)$. Noting that conjugation by $M$ preserves Fredholmness and the Fredholm index, we may establish the desired Fredholm properties
of $\T=M\circ\TT\circ M^{-1}$ acting on the weighted space $H^5_{-\eta}$ by studying the operator $\TT$ acting on the non-weighted $H^5$.  These latter properties are established by following the work \cite{Grushin}, where the author relates the pseudodifferential operator $\TT$ acting on $H^5$ to a positively homogeneous function $A$ and determining the winding
number of $A$ around the origin.  The relevant details are summarized in Appendix \ref{Fredholm-appendix}.  

By direct calculation, the symbol of $\TT$ is seen to be 
\[
\t(x, \xi) = 1 - c_0\phi_+(x)\ell(\xi-\i \eta) - c_0\phi_-(x) \ell(\xi+\i \eta),
\]
where $\phi_\pm(x) = \exp(\pm \eta x)/(2\cosh(\eta x))$.  In particular, note that 
\begin{equation}\label{phi_limits}
\left\{\begin{aligned}
\lim_{x\to\infty}\phi_+(x)&=1~~{\rm and}~~\lim_{x\to -\infty}\phi_+(x)=0\\
\lim_{x\to\infty}\phi_-(x)&=0~~{\rm and}~~\lim_{x\to -\infty}\phi_-(x)=1.
	\end{aligned}\right.
\end{equation}
 Following \cite{Grushin}, we define the positive, homogeneous degree-zero function 
\[
A(x_0, x, \xi_0, \xi) \coloneqq \tilde{t}\left(\frac{x}{x_0},\frac{\xi}{\xi_0}\right)
\]
for $x,\xi\in\R$ and $x_0,\xi_0>0$ and study $A$ acting on\footnote{Note $A$ may be extended by continuity down to $x_0=0$ and $\xi_0=0$.}
\[
\overline{\SS^1_+} \times \overline{\SS^1_+}\coloneqq\left\{(x_0,x,\xi_0,\xi)\in\R^4 \, \Big|\, x_0^2+x^2=\xi_0^2+\xi^2=1,~~x_0\geq 0,~~\xi_0\geq 0\right\}.
\] 
According to Proposition \ref{Fredholm-Grushin}, the linear opertaor $\TT$ is Fredholm provided that the function $A$ is smooth in $\overline{\SS^1_+} \times \overline{\SS^1_+}$ and
nowhere vanishing along the boundary $\Gamma$ of $\overline{\SS^1_+}\times\overline{\SS^1_+}$, which can be decomposed into the arcs
\begin{alignat*}{4}
  \Gamma_1& = \{ &(0,1, \xi_0, \xi)  & \mid\, & \xi_0^2 + \xi^2  & = 1,\, &\xi_0 &\geq 0\} \\
  \Gamma_2& = \{ & (0, -1, \xi_0, \xi)& \mid\, &  \xi_0^2 + \xi^2 & = 1,\,& \xi_0 & \geq 0\} \\
  \Gamma_3& = \{ & (x_0, x , 0, 1)    & \mid\, &  x_0^2 + x^2     & = 1,\,& x_0   & \geq 0\} \\
  \Gamma_4& = \{ & (x_0, x , 0, -1)   & \mid\, &  x_0^2 + x^2     & = 1,\,& x_0   & \geq 0\}.
  \end{alignat*}
Further, the Fredholm index of $\TT$ is precisely the winding number of $A$ as $\Gamma$ is transversed in the counter-clockwise direction, that is 
 \[
    \begin{tikzcd}
      (0,-1,0,1) \arrow{d}{\Gamma_2} & (0,1,0,1) \arrow{l}{\Gamma_3} \\%
      (0, -1,0, -1) \arrow{r}{\Gamma_4}& (0, 1, 0, -1). \arrow{u}{\Gamma_1}
    \end{tikzcd}
  \] 
As such, it is important to locate the roots of the function $1-c_0\ell$ when the parameters $\tau,c_0$ correspond to crossing the bifurcation
curves $C_1$ and $C_2$.  This  motivates the following Lemma.
%

\begin{lem}\label{zeros} 
The multiplier $\ell \colon \CC \to \CC$ is analytic in the complex strip $|{\IIm} \, z| < \eta^*$, with $\eta_*$ as in Proposition \ref{K-prop}. 
Moreover, there exists a possibly smaller strip $|{\IIm}\, z| <\eeta$ in which the function $1-c_0\ell:\CC \to \CC$ has precisely the zeros 
	 \begin{enumerate}[(i)]
   \item $k_{0}, -k_{0}, 0,$ and $0$ counting multiplicities for some $k_0>0$, when $\tau_0$ and $c_0$ are as in~\eqref{parameter-r1},
	\item $s, s, -s$ and $-s$ counting multiplicities for some $s>0$, when $\tau$ and $c_0$ are as in~\eqref{parameter-r2}.
	\end{enumerate}
See Figure~\ref{fig-symbol}.
\end{lem}

\begin{proof}See Corollary 2.2 in~\cite{EJMR} for the analyticity of $\ell$. Item (i) can be found as Lemma 2 in~\cite{AK} and item (ii) can be found in~Section IV in~\cite{Kirchgassner}.\end{proof}

\begin{figure}
	\begin{center}
		\includegraphics[scale=0.85]{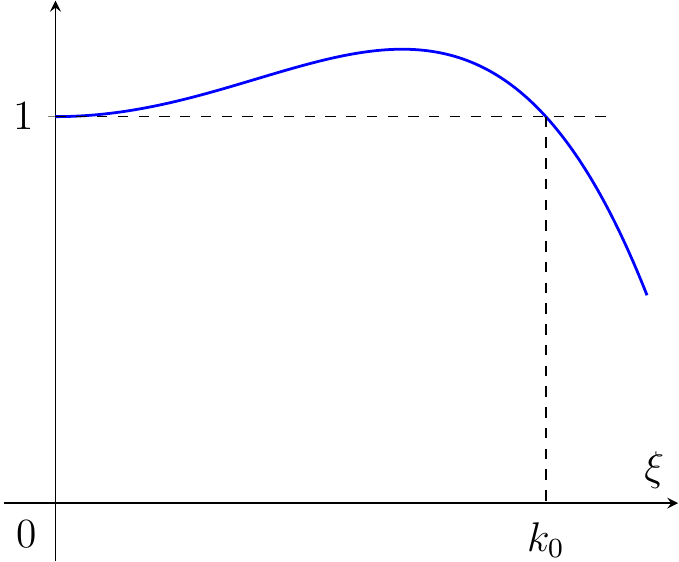} \hspace{0.7cm}
		\includegraphics[scale=0.85]{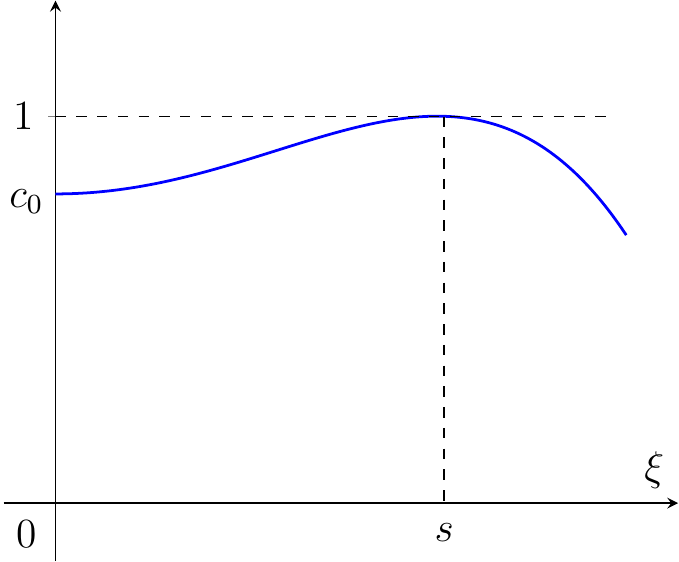}
		\captionsetup{width=.8\linewidth}
		\caption{The multiplier $c_0\ell$ for~\eqref{parameter-r1} (left) and for \eqref{parameter-r2} (right).}
		\label{fig-symbol}
	\end{center}
\end{figure}

With this preliminary result, we are now ready to prove the main result of this section.

\begin{thm} \label{Fredholm}
For each $\eta \in (0, \min\{\eta^*, \eeta\})$ and choice of parameters $\tau,c_0$ satisfying either \eqref{parameter-r1} or \eqref{parameter-r2}, 
the linear operator $\T:H^5_{-\eta} \to H^5_{-\eta}$ is Fredholm with Fredholm index four. 
For each set of parameters, its nullspace $\Ker \T$ is four-dimensional, given by
\begin{equation}\label{ker_1}
\Ker \T = \Span \{1, x, \cos(k_{0}  x), \sin(k_{0} x)\}
\end{equation}
if $\tau,c_0$ satisfy \eqref{parameter-r1}, and
\begin{equation}\label{ker_2}
\Ker \T = \Span \{\cos(s x), x \cos(sx), \sin(sx), x\sin(sx)\}
\end{equation}
if $\tau,c_0$ satisfy \eqref{parameter-r2}.
\end{thm}

\begin{proof} 
Let $\eta \in (0, \min\{\eta^*, \eeta\})$ be fixed. Following the outline above, we first verify the Fredholmness of $\TT$ by showing that $A$ is non-vanishing on $\Gamma$.
To this end, recall that
\[
\ell(z) = \left(\frac{1}{1+\tau z^2} \,\frac{z}{\tanh(z)}\right)^{1/2}, \hspace{0.5cm} z \in \CC.
\]
Along $\Gamma_1$ we have $\xi_0=\sqrt{1-\xi^2}$ and hence for $(0,1,\xi_0,\xi)\in\Gamma_1$ with $\xi_0\neq 0$, i.e. $\xi\neq \pm 1$, we have, recalling \eqref{phi_limits},
\[
A(0,1,\xi_0,\xi) = \lim_{x_0\to 0^+}\tilde{t}\left(\frac{1}{x_0},\frac{\xi}{\sqrt{1-\xi^2}}\right) = 1- c_0\ell\left(\frac{\xi}{\sqrt{1-\xi^2}}- \i\eta\right).
\]
To evaluate at the end points $(0,1,0, 1)$ and $(0,1,0,-1)$, it is equivalent to compute the limit of $\ell(\xi'- \i\eta)$ as $\xi'\to\infty$ and $\xi'\to-\infty$, respectively.  
A calculation gives
\[
\big|\ell(\xi' \pm \i\eta)\big|^4 = \frac{4 (\xi')^2+ 4\eta^2}{(1+\tau(\xi')^2 \mp \tau \eta^2)^2 + (2\tau \xi' \eta)^2} \cdot \frac{(\cosh^2(\xi')-1+\cos^2(\eta))^2}{\cosh^2(2\xi')-\cos^2(2\eta)},
\]
which implies that $|\ell(\xi'\pm \i\eta)|\to 0$ as $|\xi'|\to\infty$.  Consequently, $A(0,1,0,\pm 1)= 1$.  By similar calculations, we find
%
%
%
\[A(x_0, x, \xi_0, \xi) = \begin{cases} 1- c_0\ell\left(\frac{\xi}{\sqrt{1-\xi^2}}- \i\eta\right), & \text{on $\Gamma_1\setminus\{(0,1,0,\pm 1)\}$}\\
1- c_0\ell\left(\frac{\xi}{\sqrt{1-\xi^2}} + \i\eta\right), & \text{on $\Gamma_2\setminus\{(0,-1,0,\pm 1)\}$}\\
1, & \text{on $\Gamma_3\cup\Gamma_4$.}
\end{cases}\]
In view of Lemma~\ref{zeros}, $A$ is smooth and nowhere vanishing on $\Gamma=\cup_{j=1}^4\Gamma_j$
and hence Proposition \ref{Fredholm-Grushin} implies that $\tilde{\T}$ is a Fredholm operator, as desired.
  
\begin{figure}[t]
\begin{center}
\includegraphics{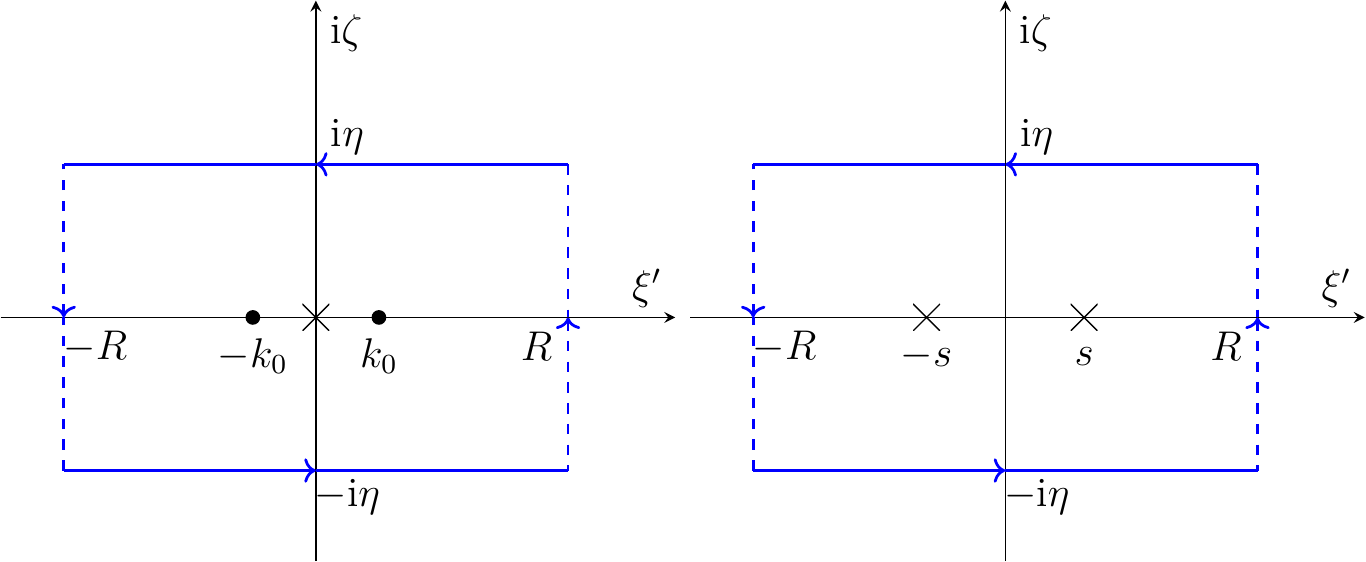}
\captionsetup{width=.8\linewidth}
\caption{The rectangular contour $\Gamma_R$ in Theorem~\ref{Fredholm}. It consists of the arcs $\{ \xi' \pm \i \eta \, : \, |\xi'| \leq R\}$ and $\{\pm R + \i \zeta \, : \,  |\zeta| \leq \eta\}$. The left picture illustrates the case~\eqref{parameter-r1} and the right illustrates~\eqref{parameter-r2}. Here, dots and crosses represent zeroes of $1-c_0\ell(\xi)$ with multiplicity one and two, respectively.}
\label{GammaR}
\end{center}
\end{figure}

Next, we compute the Fredholm index of the operator $\TT$ by computing the winding number of $A$ along $\Gamma$ transversed in the counter-clockwise direction (as described
above).   Setting $\xi'=\xi(1-\xi^2)^{-1/2}$ for $|\xi|<1$ we see that traversing from $(0,1,0,-1)$ to $(0,1,0,1)$ along $\Gamma_1$ corresponds to considering
$1-c_0\ell(\xi'-\i\eta)$ as $\xi'$ varies from $\xi'=-\infty$ to $\xi'=\infty$. while traversing from $(0,-1,0,1)$ to $(0,-1,0,-1)$ along $\Gamma_2$ corresponds
to considering $1-c_0\ell(\xi'+\i\eta)$ as $\xi'$ varies from $\xi'=\infty$ to $\xi'=-\infty$.  Further, since
$A$ is constant along $\Gamma_{3}$ and $\Gamma_4$, traversing along these arcs does not contribute to the winding number of $A$ along $\Gamma$.
To compute the winding number, choose $\tau,c_0$ to satisfy either \eqref{parameter-r1} or \eqref{parameter-r2}.  Let $R>0$ be strictly larger than the corresponding
values $k_0>0$ or $s>0$ from Lemma \ref{zeros}, and consider the rectangular contour $\Gamma_R$ with vertices at $(\pm R, \pm \i \eta)$: see Figure~\ref{GammaR}. 
The winding number of $A$ along $\Gamma$ is in fact limit of the winding number of $1-c_0\ell$ along $\Gamma_R$ as $R \to \infty$.  The latter is computed via
\[
\begin{split}\frac{1}{2\pi \i} \int_{\Gamma_R} \frac{c_0\ell'(z)}{1-c_0\ell(z)} \dif z &= \frac{1}{2\pi \i} \Big(-\int_{-R}^R \frac{c_0\ell'(\xi'+ \i\eta)}{1-c_0\ell(\xi' + \i\eta)} \dif \xi' + \int_{-R}^R \frac{c_0\ell'(\xi'- \i\eta)}{1-c_0\ell(\xi'-\i\eta)} \dif \xi'\\
&+ \int_{-\eta}^\eta \frac{c_0 \ell'(R+\i \zeta)}{1-c_0\ell(R+\i \zeta)} \dif \zeta - \int_{-\eta}^\eta \frac{c_0\ell'(-R + \i \zeta)}{1-c_0\ell(-R+\i\zeta)} \dif \zeta \Big).
\end{split}
\]
Since $1-c_0\ell$ is analytic, it follows by Lemma \ref{zeros} and the residue theorem that the above integrals sum to exactly four. Noting that the last two integrals vanish as $R \to \infty$ since $|\ell(\pm R + \i \zeta)| \to 0$ and $|\ell'(\pm R + \i \zeta)| \to 0$ as $R \to \infty$ and $|\zeta|\leq \eta$, the winding number of $A$ around $\Gamma$ is four. Proposition \ref{Fredholm-Grushin} now gives that the Fredholm index of $\TT$, and hence that of $\T$, is indeed four.

Finally, it remains to characterize the kernel of $\T$ acting on $H^5_{-\eta}$ when $\tau$ and $c_0$ satisfy either \eqref{parameter-r1} or \eqref{parameter-r2}.  Observe here that the equation $\T f=0$ with $f\in L^2_{-\eta}$ can not be studied directly by the Fourier transform since the Fourier transform of such $f$ is not a tempered distribution. To this end, we argue along the same lines as \cite[Proposition 2.9]{TWW} and consider instead the range of $\T \colon L^2_{\eta} \to L^2_{\eta}$, 
which is the adjoint of $\T \colon L^2_{-\eta} \to L^2_{-\eta}$ under the $L^2$-pairing. The Fourier transform of the range equation $\T f = g$ in $L^2_\eta$ is precisely 
\[
(1-c_0\ell(\xi))\F f(\xi) = \F g(\xi).
\]
In view of Lemma~\ref{zeros}, it follows that the range of $\T$ on $L^2_\eta$ consists of functions $g$ whose Fourier transforms
vanish on the zero set of $1-c_0\ell(\xi)$.  If $\tau,c_0$ satisfy \eqref{parameter-r1}, it follows from Lemma \ref{zeros} that the range
of $\T$ acting on $L^2_\eta$ consists of functions $g$ that satisfy
\[
\F g(0) = (\F g)'(0) = \F g(\pm k_0) = 0
\]
or, equivalently,
\[
\int_\R 1 \cdot g(x) \dif x = \int_\R x \cdot g(x) \dif x = \int_\R \exp(\mp \i k_0 x) g(x) \dif x = 0.
\]
By duality, such $\tau,c_0$ the kernel of $\T$ acting on $L^2_{-\eta}$ is given by \eqref{ker_1}, which clearly also belongs to $H^5_{-\eta}$.

Similarly, if $\tau,c_0$ satisfy \eqref{parameter-r2} then Lemma \ref{zeros} implies that the kernel of  $\T$ acting on $L^2_\eta$ consists of functions $g$ satisfying
\[
\F g(\pm s) = (\F g)'(\pm s) = 0
\]
or, equivalently,
\[
\int_\R \exp(\pm \i s x) g(x) \dif x = \int_\R x \cdot \exp(\pm \i s x) g(x) \dif x = 0.
\]
Again by duality, this shows that the kernel of $\T$ on $L^2_{-\eta}$ is given by \eqref{ker_2}, which again belongs to $H^5_{-\eta}$. 
\end{proof}

\section{Center manifold reduction}
\label{sect-CMT}
We use a nonlocal center manifold theorem, originally introduced in \cite{FS} and later adapted  in \cite{TWW} 
to account for the non-integrability of $K'$.  For completeness, the general result used here is recorded in Appendix~\ref{CMT-appendix}. 
In this section, we apply this general result to the nonlocal profile equation \eqref{eq} together with the modified equation
\begin{equation}\label{modified-gra-cap} 
\T \p + \N^\delta(\p, \mu) = 0,
\end{equation}
where
\[
\N^\delta(\p, \mu) = \N(\chi^\delta(\p), \mu)
\]
and $\chi^\delta:H^5_{-\eta} \to H^5_{\uu}$ is the nonlocal and  translationally invariant cutoff operator defined in \eqref{cutoff} in Appendix \ref{CMT-appendix}.
In particular, $\chi^\delta$ maps $\p \in H^5_{-\eta}$ to a ball of radius $C\delta$ in $H^5_{\uu}$, the space of uniform locally $H^5$ functions, with norm
\[
\|\p\|_{H^5_{\uu}} = \sup_{t \in \R}\|\p(\, \cdot \, + t)\|_{H^5([0,1])}.
\] 
More precisely, there exists a constant $C>0$ such that 
\[
\chi^\delta(\p) = \left\{\begin{aligned}
											\p&~~{\rm if}~~\|\p\|_{H^5_{\uu}}\leq C\delta\\
											0&~~{\rm if}~~\|\p\|_{H^5_{\uu}}~~{\rm is~sufficiently~large}
									\end{aligned}\right.
\]
and hence for $\|\p\|_{H^5_{\uu}} \leq C \delta$ we have $\N^\delta(\p, \mu)= \N(\p, \mu)$. Then, such small solutions of \eqref{modified-gra-cap}
are also solutions of the original profile equation \eqref{eq}.  Furthermore, note that since $H^5_{\uu}$ is continuously embedded
in $H^5_{-\eta}$ for all $\eta>0$, the operator $\chi^\delta$ also serves as a cutoff in the $H^5_{-\eta}$ norm as well.  For more details, see Appendix~\ref{CMT-appendix}.  

A central ingredient of the center manifold reduction is the construction of a bounded projection $\Q:H^5_{-\eta}\to H^5_{-\eta}$ 
onto $\Ker \T$, which could be any bounded projection having a continuous extension to $H^4_{-\eta}$ and commuting with the inclusion map 
from $H^5_{-\eta}$ to $H^5_{-\eta'}$ for all $0 < \eta' < \eta$. 
Since the nonlocal profile equation \eqref{eq} is invariant under all spatial translations, a specific choice of $\Q$ simplifies the computations significantly. 
Indeed, from Theorem \ref{Fredholm} we know $\Ker \T$ has dimension four and hence, keeping generality for the moment, we may take
\[
\Ker \T = \Span\left\{e_1,e_2,e_3,e_4\right\}
\]
for appropriately chosen, linearly independent functions $e_j$. Follow the recommendation in \cite{FS}, we aim to choose a projection $\Q \colon H^5_{-\eta} \to \Ker \T$ 
\[
\Q \colon \p \mapsto A e_1 + B e_2 + C e_3 + D e_4,
\]
which relates the coefficients $A, B, C$ and $D$ to $\p(0), \p'(0), \p''(0)$ and $\p'''(0)$ via a transition matrix $\mathscr T$
\[
\mathscr T \colon (\p(0), \p'(0), \p''(0), \p'''(0)) \mapsto (A, B, C, D).
\]
Using that $\Q \p = \Q^2 \p$, a straightforward computation yields
\[
\mathscr T = \begin{pmatrix} 
		e_1(0) & e_2(0) & e_3(0) & e_4(0) \\ 
		e_1'(0) & e_2'(0) & e_3'(0) & e_4'(0) \\ 
		e_1''(0) & e_2''(0) & e_3''(0) & e_4''(0) \\ 
		e'''_1(0) & e'''_2(0) & e'''_3(0) & e'''_4(0) 
		\end{pmatrix}^{-1}.
\]
When the parameters $\tau,c_0$ satisfy \eqref{parameter-r1}, $\Ker \T = \Span\{1, x, \cos(k_0x), \sin(k_0 x)\}$ according to Theorem \ref{Fredholm} and
the transition matrix with respect to these basis functions is
\begin{equation} \label{transition-r1} 
 \mathscr T_1= \begin{pmatrix} 
		1 & 0 & k_0^{-2} & 0 \\ 
		0 & 1 & 0 & k_0^{-2} \\ 
		0 & 0 & -k_0^{-2}& 0 \\ 
		0 & 0 & 0 & -k_0^{-3} 
		\end{pmatrix},
\end{equation}
which gives the explicit choice
\begin{equation}\label{proj-r1} 
\begin{split}\Q_1\p(x) &= \left(\p(0) + k_{0}^{-2}\p''(0)\right) + \left(\p'(0) + k_{0}^{-2} \p'''(0) \right)x \\
	& \hspace{1cm} - k_{0}^{-2} \p''(0) \cos(k_{0}x) - k_{0}^{-3} \p'''(0) \sin(k_{0}x). 
\end{split} 
\end{equation}
Similarly, when $\tau,c_0$ satisfy \eqref{parameter-r2} we have $\Ker \T = \Span\{\cos(sx), x\cos(sx), \sin(sx), x \sin(sx)\}$, and the transition matrix
with respect to these basis functions is 
\begin{equation} \label{transition-r2}
\mathscr T_2 = \begin{pmatrix}
1 & 0 & 0 & 0 \\ 0 & -1/2 & 0 & -(2s^2)^{-1} \\ 0 & 3(2s)^{-1} & 0 & (2s^3)^{-1} \\ s/2 & 0 & (2s)^{-1} & 0  
\end{pmatrix},
\end{equation}
which gives the explicit choice
\begin{equation}\label{proj-r2}
\begin{split}
\Q_2\p(x)  &=  \p(0) \cos(sx) - \left(\frac{1}{2}\p'(0) + \frac{1}{2s^2}\p'''(0)\right) x\cos(sx)  \\
& \hspace{1cm}+\left(\frac{3}{2s}\p'(0) + \frac{1}{2s^3}\p'''(0)\right) \sin(sx) + \left(\frac{s}{2}\p(0) + \frac{1}{2s}\p''(0)\right) x\sin(sx).
\end{split}
\end{equation}
\begin{rem}\label{rem-reg-choice}Our analysis up until this point holds in any space $H^m_{-\eta}$ with $m\ge 1$ and the choice of space $H^5_{-\eta}$ is made here. The projections $\Q_1$ and $\Q_2$ are required to have a continuous extension to $H^{m-1}_{-\eta}$. Because these involve pointwise evaluation of $\p'''$, we need at least $m-1=4$ which explains the choice $H^5_{-\eta}$.
\end{rem}
Lastly, the shift operator $\p \mapsto \p(\, \cdot \, + t)$ will be denoted by $S_t$.  We are in position to apply the nonlocal center manifold theorem
to equation \eqref{eq}.
 
\begin{thm}\label{CMT} 
There exist a neighborhood $\V$ of $0 \in \R$, a cutoff radius $\delta > 0$, a weight $\eta < \min\{\eta^*, \eeta\}$ and a map
\[
\Psi\colon \R^4 \times \V  \to \Ker \Q \subset H^5_{-\eta}
\]
with the center manifold
\[
\mathscr{M}^\mu_0 =\big\{Ae_1+ Be_2+ C e_3 + D e_4 + \Psi(A, B, C, D, \mu) \, \big| \, (A, B, C, D) \in \R^4 \big\}
\]
as its graph for each $\mu \in \V$. Here, $\Ker \T = \Span\{e_j\}_{j=1}^4$ and functions $e_j$ are taken to be as in Theorem~\ref{Fredholm} for
the given choices of $\tau,c_0$. The following statements hold:
  \begin{enumerate}[(i)]
    \item \label{CMT-smooth}(smoothness) $\Psi$ is $\C^4 $;
    \item \label{CMT-tan}(tangency) $\Psi(0,0,0,0,0) = 0$ and ${\D}_{(A, B, C, D)}\Psi(0,0,0,0,0)=0;$
    \item \label{CMT-glob-red}(global reduction) $\Mm^\mu_0$ consists precisely of solutions $\p \in H^5_{-\eta}$ with parameter $\mu$ to the modified equation~\eqref{modified-gra-cap};
    \item \label{CMT-loc-red}(local reduction) any $\p$ solving~\eqref{eq} with parameter $\mu$ and $\|\p\|_{H^5_{\uu}} \lesssim \delta$ is contained in $\Mm_0^\mu$;
    \item \label{CMT-corr} (correspondence) $\p \in \Mm^\mu_0$ if and only if it solves the ODE
          \begin{equation}\label{reduced-full}\p''''(t) = g(\mathscr T(\p(t), \p'(t), \p''(t), \p'''(t)), \mu), \end{equation}
	where 
	\[\begin{split} &g(A, B, C, D, \mu) \\
	&= \frac{\dif^4}{\dif x^4} \Big(Ae_1(x)+Be_2(x)+Ce_3(x)+De_4(x)+\Psi(A,B, C, D,\mu)(x)\Big)\Big|_{x=0} 
	\end{split}\]
	and $\mathscr T$ is the transition matrix $\mathscr T_1$ in~\eqref{transition-r1}, $\mathscr T_2$ in~\eqref{transition-r2} for~\eqref{parameter-r1} and~\eqref{parameter-r2}, respectively;
    \item \label{CMT-reversibility}(reversibility) equation~\eqref{eq} possesses the translation symmetries $S_t$ and a reflection symmetry $R\p(x) \coloneqq \p(-x)$, meaning 
    \[\T S \p = S \T \p \hspace{0.5cm}\text{and} \hspace{0.5cm}\N(S \p, \mu) = S\N(\p, \mu)\] 
    if $S$ is $S_t$ or $R$. Equation~\eqref{eq} is thus reversible; so is the modified equation~\eqref{modified-gra-cap} and the reduced ODE~\eqref{reduced-full} in item~(\ref{CMT-corr}). 
  \end{enumerate}
\end{thm}

\begin{proof} 
We use the nonlocal center manifold theorem from \cite{TWW}, which for completeness is stated in Theorem~\ref{abstract-CMT}.  
The Hypothesis~\ref{hyp-T} on the linear operator $\T$ has been verified in the previous section and, further, the Hypothesis \ref{hyp-N} for the nonlinearity 
$\p \mapsto \p^2$ in $H^m_{-\eta}$ for $m\ge 1$ was verified in\footnote{Technically, the hypothesis was verified in \cite{FS-corrig} in the case
$m=1$, and then later extended to $m\geq 1$ in \cite{TWW}.} \cite{FS-corrig,TWW}.  Using that convolution with $K$ is a bounded linear mapping 
on $H^5_{-\eta}$ it follows that Hypothesis \ref{hyp-N} holds for our nonlocal nonlinearity $\p \mapsto K*\p^2$ as well.
Note that the regularity $k \geq 2$ in Hypothesis~\ref{hyp-N} is arbitrary for $\N^\delta$ in~\eqref{modified-gra-cap}, 
possibly at the cost of a smaller cutoff radius $\delta$ and $\eta > 0$.   Forthcoming computations with the reduced ODEs motivate the choice of $k=4$, which we now take.
Finally, the symmetries in item (\ref{CMT-reversibility}) above are easily checked, using that equation~\eqref{eq} is steady, that the cutoff $\chi^\delta$ 
commutes with both $R$ and $S_t$, and that $K$ is an even function.  It follows that Theorem~\ref{abstract-CMT} applies directly to the present case,
giving statements (\ref{CMT-smooth})--(\ref{CMT-loc-red}) and (\ref{CMT-reversibility}). 

It remains to prove the claim in (\ref{CMT-corr}) above.  To this end, let $\p \in \Mm_0^\mu$ and note that, since Theorem~\ref{abstract-CMT}(\ref{CMT-reversibility}) 
implies that $\Mm_0^\mu$ is invariant under translation symmetries,
we have $S_t\p\in\Mm_0^\mu$ for all $t\in \R$.  Consequently, there exist functions $A(t)$, $B(t)$, $C(t)$ and $D(t)$ defined for all $t\in\R$ such that
\begin{equation}\label{trans_identity}
\begin{aligned}
S_t \p(x) = &A(t)e_1(x) + B(t)e_2(x) + C(t)e_3(x) + D(t) e_4(x) \\
& \quad+ \Psi(A(t), B(t), C(t), D(t), \mu)(x).
\end{aligned}
\end{equation}
for each $t\in \R$.  Noting that the left-hand side of~\eqref{reduced-full} 
can be rewritten as 
\[
\p''''(t) = \p''''(x + t) \Big|_{x=0} = \frac{\dif^4}{\dif x^4} S_t\p \Big|_{x=0},
\]
differentiating the identity \eqref{trans_identity} four times in $x$ and evaluating at $x=0$ yields precisely the right-hand side in \eqref{reduced-full} with $A(t)$, $B(t)$, $C(t)$ and $D(t)$.  Statement (\ref{CMT-corr}) is now proved by using the transition matrix $\mathscr T$ to rewrite \eqref{reduced-full} in terms of $\p(t)$, $\p'(t)$, $\p''(t)$ and $\p'''(t)$.

\end{proof}

\begin{rem}\label{rem-CMT}
	The above proof shows that the projection coefficients $A(t), B(t), C(t)$ and $D(t)$, defined by the shift action $S_t$ on $\p \in \Mm_0^\mu$, are in fact $H^5_{-\eta}$ functions in $t$ because 
	\[(A(t), B(t), C(t), D(t)) = \mathscr T(\p(t), \p'(t), \p''(t), \p'''(t)), \hspace{0.5cm}t \in \R.\]
	Here, $\mathscr T$ is the transition matrix $\mathscr T_1$ for~\eqref{parameter-r1} and $\mathscr T_2$ for~\eqref{parameter-r2}.  
\end{rem}

In the next sections, we will verify our main results Theorem \ref{thm-existence-GSW-intro} and Theorem \ref{thm-existence-r2-intro} by studying
the reduced ODE equations for \eqref{modified-gra-cap} for appropriate values of of the parameters $\tau$ and $c_0$.

\section{Existence of generalized solitary waves}
\label{sect-reduced-r1}

We now establish Theorem \ref{thm-existence-GSW-intro} by deriving and studying the reduced ODE for equation~\eqref{modified-gra-cap} for $\tau,c_0$ satisfying \eqref{parameter-r1}. 
First, we assume that $\p \in \Mm_0^\mu$ is so small in the $H^5_{\uu}$ norm that it is a solution of~\eqref{eq}.  
Expanding the reduced function $\Psi$ in $A, B, C, D$ and $\mu$, and then substituting into~\eqref{eq} gives the reduced ODE up to second-order 
terms. We observe from the linear part of the truncated ODE that we have a reversible $0^{2+}(\i k_0)$ bifurcation and then apply normal form theory for this bifurcation phenomenon. It turns out that equation~\eqref{reduced-full} at leading orders is almost identical to the reduced ODE for the two-dimensional gravity-capillary water wave equations in this parameter region. 
Theorem \ref{thm-existence-GSW-intro} is then established after a persistence argument.

\subsection{The reduced system}

Recall that Remark~\ref{rem-CMT} highlights how the projection coefficients $A, B, C$ and $D$ may be interpreted as differentiable functions and, further, 
Theorem~\ref{CMT}(\ref{CMT-corr}) suggests working with these rather than the $\p(t), \p'(t), \p''(t)$ and $\p'''(t)$ directly. Next, we Taylor expand the function $\Psi$ up to second-order terms to obtain the following truncated system of ODEs.

\begin{prop}\label{trunc-syst}
  Equation~\eqref{reduced-full} in terms of $A, B, C$ and $D$ is 
\begin{align} \label{reduced-syst}
\begin{dcases}
\begin{aligned}
\frac{\dif A}{\dif t} & = B \\
\frac{\dif B}{\dif t} & = \frac{1}{k_0^2}\Psi(A, B, C, D, \mu)''''(0)\\
\frac{\dif C}{\dif t} & = k_0 D\\
\frac{\dif D}{\dif t} &=  -k_0C - \frac{1}{k_0^3}\Psi(A, B, C, D, \mu)''''(0).
\end{aligned}
\end{dcases}
\end{align}
Moreover, with $\sigma = \ell''(0)^{-1} = 1/(1/3 - \tau_0)$, we have
\begin{equation}\label{psi} \begin{split} &\Psi(A, B, C, D,\mu)''''(0) = 2\sigma k_{0}^2 \mu A - \frac{2k_{0}^3}{\ell'(k_{0})} \mu C - 2\sigma k_{0}^2 A^2 + \frac{4k_{0}^3}{\ell'(k_{0})} AC \\
&\hspace{0.5cm} - \left(\frac{3\sigma^{-2} - 4\sigma^{-1} - 4/15}{3} \sigma^2 k_{0}^2 - 4\sigma \right)B^2  + \left( 2k_{0}^3\frac{\ell''(k_{0})-2\ell'(k_{0})^2}{\ell'(k_{0})^2} - \frac{10k_{0}^2}{\ell'(k_0)} \right) BD\\
&\hspace{0.5cm} + \left(\frac{8\ell(2k_{0})k_{0}^4}{\ell(2k_{0})-1} - \sigma k_{0}^2 \right)C^2 - \left(\frac{8\ell(2k_{0})k_{0}^4}{\ell(2k_{0})-1} + \sigma k_{0}^2 \right)D^2\\
&\hspace{0.5cm} + \O\left(|(A, B, C, D)| \left(\mu^2 + |A|^2+|B|^2+|C|^2 + |D|^2 \right) \right).
\end{split}\end{equation}
\end{prop}

\begin{proof} Deriving equation~\eqref{reduced-syst} from~\eqref{reduced-full} using the transition matrix $\mathscr T_1$ is straightforward.  Indeed, simply note
that \eqref{reduced-full} is equivalent to 
\[
\frac{\dif}{\dif t}\left(\begin{array}{c} A \\ B \\ C \\ D\end{array}\right) = 
	\mathscr T_1 \left(\begin{array}{cccc}0 & 1 & 0 & 0\\ 0 & 0 & 1 & 0\\ 0 & 0 & 0 & 1\\0 & 0 & 0 &0\end{array}\right)\mathscr T_1^{-1}\left(\begin{array}{c} A \\ B \\ C \\ D\end{array}\right)
	+ \left(\begin{array}{c} 0 \\ 0 \\0 \\ g(A, B, C, D,\mu)\end{array}\right).
\]
Noting that, in this case, 
\[
g(A,B,C,D,\mu) = k_0^4C + \Psi(A,B,C,D,\mu)''''(0),
\]
a direct calculation shows that the above is precisely \eqref{reduced-syst}.  

It remains to compute the asymptotic expansion \eqref{psi}.
Specifically, we focus on computing the function $\Psi(A, B, C, D, \mu)''''$ 
evaluated at $x=0$ up to order two in $A, B, C, D$ and $\mu$. According to Theorem~\ref{CMT}(\ref{CMT-smooth}), $\Psi$ is $\C^4$ in 
$(A, B, C, D,\mu)$.   Together with item (\ref{CMT-tan}) in Theorem~\ref{CMT} and the fact that $\p \equiv 0$ is a solution to~\eqref{eq} 
for all $\mu \in \R$, it follows that the Taylor expansion of $\Psi$ must be of the form
\[
\begin{split}\Psi(A, B, C, D, \mu)(x) &= \sum_{ \substack{2\leq p+q+l+m+n \leq 3 \\ n \ge 1}} \Psi_{pqlmn}(x) \cdot  A^pB^qC^lD^m \mu^n+\cdots \end{split}
\]
where each $\Psi_{pqlmn} \colon \R \to \R$ belongs to $\Ker \Q_1 \subset H^5_{-\eta}$. It thus remains to compute
$\Psi_{pqlmn}''''(0)$ for $p+q+l+m+n=2$ and $n \ge 1$.
To this end, let $\p \in \Mm_0^\mu$ and $\|\p\|_{H^5_{\uu}} \lesssim \delta$ and note by Theorem~\ref{CMT}(\ref{CMT-loc-red}) that $\p$ solves~\eqref{eq}. 
To conveniently group the terms, we rewrite the left-hand side of equation~\eqref{eq} to have
\[
\T \p + (\Id - \T)(\p^2 - \mu \p) = 0.
\]
Since $\p$ belongs to $\Mm_0^\mu$, we know that $\p(x) = \Q_1\p(x) + \Psi(A,B,C, D,\mu)(x)$, and plugging this into the above equation gives
\[
\T (\Q_1\p + \Psi) + (\Id - \T)\Big((\Q_1 \p + \Psi)^2 - \mu (\Q_1 \p + \Psi)\Big) = 0.
\]
Using that $\T \Q_1 \p = 0$ by definition, the above can be rearranged as
\[
\T\Psi + (\Q_1\p)^2 - \mu \Q_1 \p - \T (\Q_1 \p)^2 =-\left(\Id-\T\right)\left(2\left(\Q_1\varphi\right)\Psi + \Psi^2 - \mu\Psi\right),
\]
where we note the right hand side above consists of all terms that are at least cubic in $(A,B,C,D,\mu)$. Linear equations for the functions $\Psi_{pqlmn}$ can now be read off easily, and are recorded in  Appendix~\ref{center-mani-coeff-r1}. Note that by the condition $\Q_1 \Psi_{pqlmn} = 0$, these coefficient functions $\Psi_{pqlmn}$ are uniquely
determined.  Indeed, as seen in Appendix \ref{center-mani-coeff-r1}, if there are two solutions $\Psi_{pqlmn}$ and $\tilde\Psi_{pqlmn}$, then their difference must belong to $\Ker\T\cap\Ker\Q_1$, and hence must be zero.  Further, we observe that symmetries can be used to greatly simplify the necessary computations. 
Indeed, note that the basis functions $1, x, \cos(k_0x)$ and $\sin(k_0x)$ are either even or odd and that the operators $\T$, $\Id-\T$ and $\Q_1$ map even to even and odd to odd functions. 
Consequently, as seen in Appendix \ref{center-mani-coeff-r1} the linear equations for $\Psi_{pqlmn}$ involve either only even or odd functions
and hence the solutions $\Psi_{pqlmn}$ are also necessarily  either even or odd functions. Since only even functions $\Psi_{pqlmn}$ contribute to $\Psi(A, B, C, D, \mu)''''$ evaluated at $0$, equations for odd $\Psi_{pqlmn}$ may be disregarded. The computations for $\Psi_{pqlmn}$ are detailed in Appendix~\ref{center-mani-coeff-r1} and these give equation~\eqref{psi}.
\end{proof}

\subsection{Normal form reduction}
\label{normal-form-r1}

We use normal form theory to study the reduced system \eqref{reduced-syst}, which can be written as
\begin{equation}\label{diff-eq-r1}
\frac{\dif U}{\dif t} = \Ll U + \Rr(U, \mu),
\end{equation}
where $U = (A, B, C, D)$, and $\Ll$ is precisely the linearization of \eqref{reduced-syst} about the origin, that is,
\[
\Ll = \begin{pmatrix}0 & 1 & 0 & 0 \\ 0 & 0 & 0 & 0 \\ 0 & 0 & 0& k_0\\ 0 & 0 & -k_0 & 0\end{pmatrix},
\]
and $\Rr$ is $\C^4$ in a neighborhood of $(0,0) \in \R^4 \times \R$, satisfying $\Rr(0,0) = 0$ and ${\D}_U\Rr(0,0) = 0$.  The spectrum of $\Ll$ consists of the algebraically double and geometrically simple eigenvalue $0$, as well as the pair of simple purely imaginary
eigenvalues $\pm \i k_0$.  Further, we note that the reflection symmetry $R\p(x) = \p(-x)$ on $\Ker \T$ with respect to the basis functions $1, x, \cos(k_0 x)$ and $\sin(k_0x)$ is
\[
A + B(-x) + C \cos(-k_0x) + D\sin(-k_0x) = A - Bx + C \cos(k_0 x) - D \sin(k_0 x).
\]
This shows that $R$ restricted to $\Ker \T$ is a linear mapping on $\R^4$, given by 
\[
R \colon (A, B, C, D) \mapsto (A, -B, C, -D)
\]
and, clearly, $R^2 = \Id$.  Further, by Theorem~\ref{CMT}(\ref{CMT-reversibility}), $R$ anticommutes with $\Ll$ and $\Rr$, that is, $R\Ll U = -\Ll RU$ and $R\Rr(U, \mu) = -\Rr(RU, \mu)$. 
Taken together, it follows that it is natural to expect that the origin undergoes a reversible $0^{2+}(\i k_0)$ bifurcation for parameters $\mu$ sufficiently small.  In this section,
we use the correspoding normal form theory for such bifurcations from \cite[Chapter 4.3.1]{HI} to study \eqref{diff-eq-r1} near the origin for $\mu$ sufficiently small.

To begin, we note that the eigenvectors and generalized eigenvectors of $\Ll$ are given by
\[
\xi_0 = \begin{pmatrix}1 \\ 0 \\ 0 \\ 0 \end{pmatrix}, \, \, \xi_1 = \begin{pmatrix} 0 \\ 1 \\ 0 \\ 0 \end{pmatrix}, \, \, \text{and} \, \,\zeta = \begin{pmatrix} 0 \\ 0 \\ 1 \\ \i \end{pmatrix}, 
\]
which are readily seen to satisfy
\begin{equation}\label{L-R-basis} 
\begin{aligned}
\Ll\xi_0 &= 0, & \Ll\xi_1&=\xi_0, & \Ll\zeta &= \i k_0 \zeta,\\
 R \xi_0 &= \xi_0, & R \xi_1 &= -\xi_1, & R\zeta &= \overline{\zeta}.
\end{aligned}
\end{equation}
Based on the structure of ${\bf L}$, throughout the remainder of this section $\R^4$ will be identified with $\R^2 \times \widetilde{\R^2}$ 
where $\widetilde{\R^2} \coloneqq \{(C, \overline{C}) \, \colon \, C \in \CC\}$. We are now in the position to directly apply the normal form
result \cite[Lemma 3.5]{HI}.  This result implies that there exist neighborhoods $\mathcal{V}_1$ and $\mathcal{V}_2$ 
of $0 \in \R^2 \times \widetilde{\R^2}$ and $0 \in \R$, respectively, and a polynomial change of variables 
\begin{equation}\label{change_of_var1}
U = \Aa \xi_0 + \Bb \xi_1 + \cC \zeta + \Cc \overline{\zeta} + \Phi(\Aa, \Bb, \cC, \Cc, \mu)
\end{equation}
defined in $\mathcal{V}_1$ and $\mathcal{V}_2$, which transforms the reduced system \eqref{reduced-syst} into the normal form 
\begin{equation}\label{normal-form-r1-general}
\begin{dcases}
\frac{\dif \Aa}{\dif t} = \Bb \\
\frac{\dif \Bb}{\dif t} = P(\Aa, |\cC|^2, \mu) + \rho_{\Bb}(\Aa, \Bb, \cC, \Cc, \mu) \\
\frac{\dif \cC}{\dif t} = \i k_0 \cC + \i \cC Q(\Aa, |\cC|^2, \mu) + \rho_{\cC}(\Aa, \Bb, \cC, \Cc, \mu),
\end{dcases}
\end{equation}
where $P$ and $Q$ are polynomials of degree two and one in $(\Aa, \Bb, \cC, \Cc)$, respectively. Here, the function $\Phi$ is $\C^4$, satisfying 
\[
\begin{split}\Phi(0,0&, 0,0,0) = 0, \hspace{0.5cm} \partial_{(\Aa, \Bb, \cC, \Cc)}\Phi(0,0,0,0,0) = 0\\
&\Phi(\Aa, -\Bb, \Cc, \cC, \mu) = R\Phi(\Aa, \Bb, \cC, \Cc, \mu)
\end{split}
\] 
while the remainders $\rho_{\Bb}$ and $\rho_{\cC}$ are $\C^4$ with
\[
|\rho_{\Bb}(\Aa, \Bb, \cC, \Cc, \mu)| + |\rho_{\cC}(\Aa, \Bb, \cC, \Cc, \mu)| = o((|\Aa|+|\Bb|+|\cC|)^2).
\]
For proof and more details, see \cite[Chapter 4.3.1]{HI}. 

Let
\[
\begin{split}
&P(\Aa, |\cC|^2, \mu) = p_0\mu + p_1\mu {\Aa} + p_2\Aa^2 + p_3|\cC|^2 \\
&Q(\Aa, |\cC|^2, \mu) = q_0 \mu + q_1 \Aa.
\end{split}
\]
The scalar coefficients $p_0, p_1, p_2, p_3, q_0$ and $q_1$ are computed in Appendix~\ref{appendix-normal-form-r1}. Setting $\sigma = (1/3-\tau)^{-1}$
these calculations yield the normal form of~\eqref{reduced-syst} as
\begin{equation}
\label{normal-form-r1-final}
\begin{dcases}
\frac{\dif \Aa}{\dif t} = \Bb \\
\frac{\dif \Bb}{\dif t} = 2\sigma\mu \Aa - 2\sigma \Aa^2 - 4\sigma|\cC|^2 + \O(|\mu|^2 + (|\mu| + |\Aa| + |\cC|)^2))\\
\frac{\dif \cC}{\dif t} = \i k_0 \cC + \frac{\i}{\ell'(k_0)} \mu \cC - \frac{2\i}{\ell'(k_0)} \Aa \cC + \O(|\cC|(|\mu|+|\Aa| + |\cC|^2)^2).
\end{dcases}
\end{equation}

\subsection{Generalized solitary waves}
\label{sol-existence-r1}

Next, we consider the normal form system \eqref{normal-form-r1-final} truncated at second-order terms, i.e.
\begin{equation}\label{normal-form-r1-trunc}
\begin{dcases}
\frac{\dif \Aa}{\dif t} = \Bb \\
\frac{\dif \Bb}{\dif t} = 2\sigma \mu \Aa - 2\sigma \Aa^2 - 4\sigma |\cC|^2 \\
\frac{\dif \cC}{\dif t} = \i k_0 \cC - \frac{\i}{\ell'(k_0)} \mu \cC + \frac{2\i}{\ell'(k_0)}\Aa \cC.
\end{dcases}
\end{equation}
The change of variables
\begin{equation}\label{rescalings1}
\begin{split}
&t= \frac{1}{\sqrt{2}}w, \hspace{0.5cm} \Aa(t) = -\frac{3}{2}\widetilde{\Aa}(w), \hspace{0.5cm}\Bb(t) = -\frac{3}{\sqrt{2}}\widetilde{\Bb}(w), \\
&\hspace{2cm}\cC(t) = |\mu|k^{1/2} \exp(\i \Theta(t))
\end{split}
\end{equation}
transforms~\eqref{normal-form-r1-trunc} into the system (3.14) studied by Iooss \& Kirchg\"{a}ssner in~\cite{IK} in their search for generalized solitary waves in the context
of the full gravity-capillary water wave problem.  The only difference between our rescaled system and that studied in \cite{IK} is the coefficients of terms involving $\cC$, which is inconsequential.  Note in \cite{IK} that the small parameter used is $\frac{1}{c^2}-1$, which corresponds to $-\mu$ in our case.  Here, $k$ is fixed but arbitrary. We observe that 
\[\Theta'(t)= k_0 - \frac{\mu}{\ell'(k_0)}+\frac{2}{\ell'(k_0)}\Aa(t).\]
Equations (3.17)--(3.19) in \cite{IK} provide us with a one-parameter family of explicit solutions, parametrized by $k$, of the rescaled truncated system given by
\begin{equation}\label{sol-scaled-r1}
\begin{split}&\widetilde{\Aa}(w) = -\frac{\mu}{3}(1-
\text{sgn}(\mu)\rho^{1/2})-|\mu|\rho^{1/2}\text{sech}^2\left(\frac{\rho^{1/4}|\mu|^{1/2}\sigma^{1/2} w}{2}\right), \\
&\widetilde{\Bb}(w) = \widetilde{\Aa}'(w).
\end{split}
\end{equation}
Then, substituting $\Aa(t)=-3\widetilde{\Aa}(w)/2$ into the differential equation for $\Theta(t)$ gives
\[\begin{split}&\Theta(t)= \Theta_*+ \left(k_0 - \frac{\mu}{\ell'(k_0)}+\frac{2\mu}{\ell'(k_0)}\left(1-\text{sgn}(\mu)\rho^{1/2}\right)\right)t\\
&\hspace{1cm}+\frac{3\sqrt{2}\,\rho^{1/4}|\mu|^{1/2}}{\sigma^{1/2}\ell'(k_0)}\tanh\left(\frac{\rho^{1/4}|\mu|^{1/2}\sigma^{1/2}t}{\sqrt{2}}\right),
\end{split}\]
where $\Theta_* \in \R \setminus 2\pi\ZZ$ is arbitrary and $\rho = 1 +24k.$

It remains to see if the above family of solutions of the rescaled truncated normal-form system persist as solutions of the full rescaled normal-form system.  Luckily, the persistence of~\eqref{sol-scaled-r1} under reversible perturbations has received considerable treatment (see, for example, the work of Iooss \& Kirchg\"assner in \cite{IK}). 
In particular, these persistence results are summarized for $\C^m$ vector fields in \cite[Theorem 3.10]{HI} and, in the present context, this work guarantees
that the family of explicit solutions \eqref{sol-scaled-r1} persists provided that
\begin{equation} \label{persistence}
r = |\mu|k^{1/2} > r_*(\mu) = \O(|\mu|^{1/2}).
\end{equation}
In particular, note that since $\mu$ is small the persistence condition \eqref{persistence} is effectively a lower bound on the frequency $k$, corresponding
to high-frequency oscillation in $\tilde{\cC}$. 

Finally, we undo the above variable changes to return to the original unknown function $\p$.  Undoing \eqref{rescalings1}
in \eqref{sol-existence-r1} yields
\[
\begin{split}
&\Aa(t) = \frac{\mu}{2}(1-\text{sgn}(\mu)\rho^{1/2}) + \frac{3}{2}|\mu|\rho^{1/2} \text{sech}^2\left(\frac{\rho^{1/4}|\mu|^{1/2}\sigma^{1/2}t}{\sqrt{2}}\right),\\
&\Bb(t) = \Aa'(t), \\
&\cC(t) = |\mu|k^{1/2} \exp\left(\i (k_0+\O(\mu))t + \i \Theta_* + \O(\mu)\right),
\end{split}
\]
while undoing the polynomial change of variables \eqref{change_of_var1} in the above normal form analysis yields
\begin{equation}\label{cov}
\begin{split}
&A(t) = \Aa(t) + \O(\mu^2\rho^{1/2}), \hspace{0.2cm} B(t) = \Bb(t) + \O(\mu^2\rho^{1/2}),\\ 
&\hspace{1.75cm} C(t) = \frac{1}{2}(\cC + \Cc) + \O(\mu^2\rho^{1/2}).
\end{split}
\end{equation}
Recalling now that \eqref{proj-r1} implies $A(t) = \p(t)+k_0^{-2} \p''(t)$, $C(t) = -k_0^{-2}\p''(t)$ and switching back to the original variable $x$, it follows that
\[
\begin{split}  
\p(x) &= A(x)+C(x) \\
&=  \frac{3}{2}|\mu|\rho^{1/2} \text{sech}^2\left(\frac{\rho^{1/4}|\mu|^{1/2}\sigma^{1/2}x}{\sqrt{2}}\right) + \frac{\mu}{2}(1-\text{sgn}(\mu)\rho^{1/2}) \\
&\hspace{1cm} +|\mu|k^{1/2} \cos\Big( (k_0+\O(\mu))x + \Theta_* + \O(\mu)\Big) + \O(\mu^2\rho^{1/2}).
\end{split}
\]
Here, $\Theta_* \in \R/2\pi \ZZ$ is an arbitrary integration constant. Due to the hyperbolic tangent in $\Theta$, there is an asymptotic phase shift in the cosinus term between $x = -\infty$ and $x = \infty$ of order $\O(\rho^{1/4}|\mu|^{1/2})$. 
 
Provided the persistence condition \eqref{persistence} holds, the function $\phi$ above solves the modified profile equation \eqref{modified-gra-cap}. 
For $\p$ to be a solution to the original profile equation \eqref{eq} with parameter $\mu$, it must additionally satisfy
the smallness assumption $\|\p\|_{H^5_{\uu}} \lesssim \delta$. This can be achieved by setting, for example,
\[
k = k'|\mu|^{-1-2\kappa}, \hspace{0.5cm} \text{for some $\kappa \in [0, 1/2)$ and some constant $k' > 0$}.
\]
Indeed, under this condition the persistence condition \eqref{persistence} is clearly met and
the functions $\Aa, \Bb, \cC, \Cc$ have amplitude $\O(|\mu|^{1/2-\kappa})$ which, in turn, implies that $\p, \p', \p'', \p'''$ are also $\O(|\mu|^{1/2-\kappa})$ via~\eqref{cov} and~\eqref{proj-r1}. 
This bound is carried over to the fourth and fifth derivatives by differentiating~\eqref{reduced-full} twice (see \cite[Theorem 3.3]{TWW}). 
It follows from choosing $\mu$ sufficiently small that the $H^5_{\uu}$ norm of $\p$ is small, and hence that $\p$ is a solution to~\eqref{eq} with parameter $\mu$. 
This establishes Theorem \ref{thm-existence-GSW-intro}.

\begin{rem}\label{rem-supercritical-GSW} When $\mu > 0$,~\eqref{normal-form-r1-trunc} features an orbit which is homoclinic to the saddle equilibrium $(\Aa, \Bb)=(0,0)$ once projected onto the $(\Aa, \Bb)$-plane. When $\mu < 0$, it is homoclinic to $(\Aa, \Bb) = (\frac{\mu}{2}(1+\rho^{1/2}), 0)$, which is close to $(0,0)$. In the latter case, we point out that this solution has supercritical wave speed. Indeed, equation~\eqref{gra-cap-Whitham} is invariant under a Galilean change of variable
	\[\p \mapsto \p + v, \hspace{0.3cm} c \mapsto c-2v, \hspace{0.3cm} (1-c)^2 b \mapsto (1-c)^2b+(1-c)v+v^2,\]
where $b$ is an integration constant which doesn't affect the critical wavespeed: see \cite{HJ}. Putting $v = \mu/2(1+\rho^{1/2})$, the new wave speed is 
\[c-2v = 1+\mu -2\cdot\frac{\mu}{2}(1+\rho^{1/2})= 1 + |\mu|\rho^{1/2} > 1. \]
To summarize, all generalized solitary-wave solutions in Theorem~\ref{thm-existence-GSW-intro} have supercritical wave speed $c > 1$. 
\end{rem}

\section{Existence of modulated solitary waves}
\label{sect-reduced-r2}
The aim of this section is to prove existence of modulated solitary waves in~\eqref{parameter-r2}. 
As for the classical two-dimensional gravity-capillary water wave equations, the signs of two terms in the normal form are to be determined -- one of 
those terms will be of cubic order.  Instead of deriving the full reduced ODE as in Section~\ref{sect-reduced-r1}, we only determine it roughly using the symmetries.  We then perform a normal form reduction and determine linear equations for the relevant normal form coefficients. From these, it will be clear which center manifold coefficients are necessary.  Throughout this section, we assume that the parameters
$\tau$ and $c_0$ satisfy \eqref{parameter-r2}.

\subsection{Normal form reduction}

As in Section \ref{sect-reduced-r1}, we will work with projection coefficients $A$, $B$, $C$, and $D$ rather than $\p, \p', \p''$ and $\p'''$. 
Using the transition matrix $\mathscr T_2$ from Section~\ref{sect-CMT} and proceeding along the same lines as the proof of Proposition \ref{trunc-syst}, we find
that \eqref{reduced-full} in this case is equivalent to the system
\begin{equation}\label{red-syst-r2}
\begin{dcases}\frac{\dif A}{\dif t} = B + sC\\
\frac{\dif B}{\dif t} = sD - \frac{1}{2s^2} \Psi(A, B, C, D, \mu)''''(0) \\
\frac{\dif C}{\dif t} = -sA+D + \frac{1}{2s^3}\Psi(A, B, C, D, \mu)''''(0)\\
\frac{\dif D}{\dif t} = -sB.
\end{dcases}
\end{equation}
Letting $U = (A, B, C, D)$, \eqref{red-syst-r2} can be rewritten as
\begin{equation}\label{diff-eq-r2}
\frac{\dif U}{\dif t} = \Ll U + \Rr(U,\mu),
\end{equation}
where here
\[
\Ll = \begin{pmatrix}
0 & 1 & s & 0 \\ 0 & 0 & 0 & s \\ -s & 0 & 0 & 1 \\ 0 & -s & 0 & 0
\end{pmatrix} \hspace{0.5cm}\text{and} \hspace{0.5cm} \Rr(U, \mu) = \frac{1}{2s^3}\begin{pmatrix}
0 \\ -s\Psi(U, \mu)''''(0) \\ \Psi(U, \mu)''''(0) \\ 0
\end{pmatrix}.
\]
In particular, in view of Theorem~\ref{CMT}(\ref{CMT-tan}) the matrix $\Ll$ is precisely the linearization of \eqref{red-syst-r2} about the trivial solution $0\in\R^4$.  
The spectrum of $\Ll$ is readily seen to consist of a pair of algebraically double and geometrically simple eigenvalues $\i s$ and $-\i s$ with corresponding eigenvectors and generalized eigenvectors 
\[
\zeta_0 = \begin{pmatrix}1\\ 0\\ \i\\ 0\end{pmatrix}, \quad \overline{\zeta_0} = \begin{pmatrix} 1 \\ 0 \\-\i \\0 \end{pmatrix}, \quad \zeta_1=\begin{pmatrix} 0\\ 1 \\ 0 \\\i\end{pmatrix} \quad \text{and} \quad \overline{\zeta_1} = \begin{pmatrix}0\\ 1\\ 0\\ -\i\end{pmatrix},
\]
that satisfy 
\[
\begin{split}&(\Ll-\i s)\zeta_0 = 0, \hspace{0.5cm} (\Ll-\i s)\zeta_1 = \zeta_0,\\
&(\Ll+\i s)\overline{\zeta_0} = 0, \hspace{0.5cm}(\Ll + \i s)\overline{\zeta_1} = \overline{\zeta_0}.
\end{split}
\]
As such, it is natural to expect that the system undergoes an $(\i s)^2$ bifurcation.

To analyze this bifurcation, observe that the set $\{\zeta_0, \zeta_1, \overline{\zeta_0}, \overline{\zeta_1}\}$ spans $\widetilde{\R^2} \times \widetilde{\R^2} \sim \R^4$, and  that the 
reversible symmetry $R\p(x) = \p(-x)$ restricted on $\Ker \T$ takes the form 
\[
R\colon (A, B, C, D) \mapsto (A, -B, -C, D)
\]
with respect to the basis $\{\cos(sx), x\cos(sx), \sin(sx), x\sin(sx)\}$. Furthermore, the vectors $\zeta_0$ and $\zeta_1$ also satisfy
\[
R\zeta_0 = \overline{\zeta_0} \hspace{0.5cm}\text{and}\hspace{0.5cm}R\zeta_1 =-\overline{\zeta_1}.
\]
Normal form theory for $(\i s)^2$ bifurcations now asserts that there exists a polynomial change of variable 
\[
U = \Aa \zeta_0 + \Bb \zeta_1 + \overline{\Aa \zeta_0} + \overline{\Bb \zeta_1} + \Phi(\Aa, \Bb, \overline{\Aa}, \overline{\Bb}, \mu),
\]
where $\Phi$ is a polynomial in $(\Aa, \Bb, \overline{\Aa}, \overline{\Bb})$ of degree 3, that transforms~\eqref{diff-eq-r2} into the normal form 
\begin{equation}
\label{normal-form-r2-gen}
\begin{dcases}
\frac{\dif \Aa}{\dif t} = \i s \Aa + \Bb + \i \Aa P\left(|\Aa|^2, \frac{\i}{2}(\Aa \overline{\Bb}- \overline{\Aa}\Bb)\right) + \rho_{\Aa}(\Aa, \Bb, \overline{\Aa}, \overline{\Bb}, \mu)\\
\frac{\dif \Bb}{\dif t} = \i s \Bb + \i \Bb P\left(|\Aa|^2, \frac{\i}{2}(\Aa \overline{\Bb}- \overline{\Aa}\Bb)\right) \\
\hspace{1cm}+ \Aa Q\left(|\Aa|^2, \frac{\i}{2}(\Aa \overline{\Bb}- \overline{\Aa}\Bb)\right) + \rho_{\Bb}(\Aa, \Bb, \overline{\Aa}, \overline{\Bb}, \mu).
\end{dcases}
\end{equation}
Here, the polynomials $P$ and $Q$ have degree $2$ in $(\Aa, \Bb, \overline{\Aa}, \overline{\Bb})$.  For details, see \cite[Section 4.3.3]{HI} and, specifically, Lemma 3.17 in that reference.

\subsection{Modulated solitary waves}
We now aim to consider the normal form system \eqref{normal-form-r2-gen} truncated at second-order terms. Let
\begin{equation}\label{PQexpansion-r2}
\begin{aligned}
&P\left(|\Aa|^2, \frac{\i}{2}(\Aa \overline{\Bb}- \overline{\Aa}\Bb)\right) = p_0 \mu + p_1 |\Aa|^2 + \frac{\i p_2}{2}(\Aa \overline{\Bb}- \overline{\Aa}\Bb),\\
&Q\left(|\Aa|^2, \frac{\i}{2}(\Aa \overline{\Bb}- \overline{\Aa}\Bb)\right) = q_0 \mu + q_1|\Aa|^2 + \frac{\i q_2}{2}(\Aa \overline{\Bb}- \overline{\Aa}\Bb).
\end{aligned}
\end{equation} 
The coefficients $q_0$ and $q_1$ are computed in Appendices~\ref{center-mani-coeff-r2} and~\ref{appendix-normal-form-r2}, 
\[
\begin{split}&q_0 = \frac{2}{c_0^2\ell''(s)} \quad \text{and} \quad q_1 = \frac{4(-c_0+\ell(2s)^{-1})^{-1}+8(1-c_0)^{-1}}{c_0^2 \ell''(s)}.
\end{split}
\]
One can check that this agrees with the formulas given in Theorem~\ref{thm-existence-r2-intro} by using that $m(s)=\ell(s)^{-1}$, $\ell(s)=c_0^{-1}$ and $\ell'(s)=0$. Moreover, $q_0$ and $q_1$ are both negative  because $c_0 < 1$, $c_0\ell(2s)< 1$ while $\ell''(s)<0$ for each $s > 0$ as illustrated in Figure~\ref{fig-symbol}.  Recalling that $\mu<0$ in this case, the above puts \eqref{normal-form-r2-gen} into the subcritical case considered by Iooss \& P\`erou\'eme in \cite[Section IV3]{IP}. Through the change of variables 
\[\Aa(t)=r_0(t) \exp(\i(s t + \Theta_0(t))), \hspace{0.5cm} \Bb(t)=r_1(t) \exp(\i(st+\Theta_1(t))),\]
the normal form truncated at third order terms has explicit homoclinic solutions
\[\begin{split} &r_0(t) = \sqrt{\frac{-2 q_0 \mu}{q_1}} \text{sech}(\sqrt{q_0\mu} \, t), \\
	&r_1(t) = |r_0'|, \\
	&\Theta_0(t) = p_0 \mu t - \frac{2p_1\sqrt{q_0 \mu}}{q_1} \tanh(\sqrt{q_0\mu}\,t)+\Theta_*,\\
	&\Theta_1 - \Theta_0 \in \{0, \pi\},
\end{split}\]
(see~\cite[pp.217--223]{HI}). Here, $p_0, p_1$ are as in~\eqref{PQexpansion-r2} and $\Theta_* \in \R$ is an arbitrary integration constant, resulting in a full circle of homoclinic solutions. However, only two distinct homoclinic solutions persist under reversible perturbation, when $\Theta_* = 0$ and $\Theta_* = \pi$. Tracing back to the original unknown $\p$ and variable $x$, we get 
\[\p(x)=\sqrt{\frac{-8 q_0 \mu}{q_1}} \text{sech}(\sqrt{q_0\mu} \, x) \cos\left(sx + \O\left(|\mu|^{1/2}\right)\right) + \O(\mu^2),\]  
and
\[\p(x)=-\sqrt{\frac{-8 q_0 \mu}{q_1}} \text{sech}(\sqrt{q_0\mu} \, x) \cos\left(sx + \O\left(|\mu|^{1/2}\right)\right) + \O(\mu^2).\]
The first solution $\p$ is often referred to as a modulated solitary wave of elevation and the latter is a modulated solitary wave of depression. We illustrate the elevation case in Figure~\ref{GSW}. Due to the hyperbolic tangent, there is an asymptotic phase shift of order $\O(|\mu|^{1/2})$ between $x=-\infty$ and $x=\infty$. Lastly, it can be shown that $\p, \p', \p''$ and $\p'''$ are of order $\O(|\mu|^{1/2})$ by arguing as in the previous section. The uniform locally Sobolev norm $\|\p\|_{H^5_{\uu}}$ can thus be made arbitrarily small, qualifying these as solutions to~\eqref{eq} with parameters~\eqref{parameter-r2}.  This establishes Theorem~\ref{thm-existence-r2-intro}.

\section*{Acknowledgements}
The work of M. A. Johnson was partially funded by the Simons Foundation Collaboration grant number 714021.  
T. Truong gratefully acknowledges the support of the Swedish Research Council, grant no. 2016-04999. Also, Truong appreciates the discussions and support she has received from her supervisor Erik Wahlén. Finally, the authors thank the reviewers for their careful reading, insightful comments and suggestions. 

\appendix

\section{Fredholm theory for pseudodifferential operators}
\label{Fredholm-appendix}

In this appendix, we review a Fredholm theory for pseudodifferential operators developed by Grushin in \cite{Grushin}.  This theory is applied in Section \ref{sect-lin-op} to determine the Fredholm
properties of the linear operator $\T$.

Let $x^* = (x_0, x) \in \R \times \R^n$ and
\[X^* = \{ x^* \in \R^{n+1} \, | \, x_0 \geq 0, x^* \neq 0\}.\]
Similarly, let $\xi^* = (\xi_0, \xi) \in \R \times \R^n$ and
\[E^* = \{ \xi^* \in \R^{n+1} \, |\, \xi_0 \geq 0, \xi^* \neq 0\}.\]
Let $\mathcal{A}$ be the class of functions $A(x^*, \xi^*) \in C^\infty(X^* \times E^*)$ such that $A$ is positive-homogeneous of degree 0 in $x^*$ and $\xi^*$, that is,
\[A(\lambda x^*, \xi^*) = A(x^*, \lambda \xi^*) = A(x^*, \xi^*), \hspace{0.5cm} \lambda > 0.\]
Let $\SS^n_+$ denote the hemisphere $|x^*| = 1$ and $x_0 > 0$, or $|\xi^*| = 1$ and $\xi_0 > 0$. Let $\overline{\SS^n_+}$ denote the relative closure of $\SS^n_+$ in $X^*$, or in $E^*$, that is, $\overline{\SS^n_+}$ is the hemisphere $|x^*| = 1$ (or $|\xi^*| = 1$), $x_0 \geq 0$ (or $\xi_0 \geq 0)$. Clearly, each $A \in \mathcal{A}$ is uniquely determined by its values on $\SS^n_+ \times \SS^n_+$. Conversely, each function $\tilde{A} \in C^\infty(\overline{\SS^n_+} \times \overline{\SS^n_+})$ can be uniquely homogeneously extended to $X^* \times E^*$. So, $\mathcal{A} \cong C^\infty(\overline{\SS^n_+} \times \overline{\SS^n_+})$.
By $S^0_\mathcal{A}$, we denote the set of symbols $p_A(x, \xi)$ which are given by
\[p_A(x, \xi) = A(1, x, \, 1, \xi),\]
for some $A \in \mathcal{A}$. For $p_A \in S^0_\mathcal{A}$, we have the following result, which combines Theorems~4.1 and~4.2 in Grushin~\cite{Grushin}.
\begin{thm}\label{Fredholm-Grushin} If $p_A(x, \xi) \in S^0_\mathcal{A}$ and $\det A(x^*, \xi^*) \neq 0$ on $\Gamma$, then
  \[p_A(x, D)\colon H^s \to H^s\]
  is Fredholm and the index is
  \[\ind  p_A(x, D) = \frac{1}{2\pi} \Big(\arg \det\, A(x^*, \xi^*)\Big|_{\Gamma}\Big),\]
  where $\Gamma$ is the boundary of $\overline{\SS^n_+} \times \overline{\SS^n_+}$, and $\arg \det \,A(x^*, \xi^*)|_\Gamma$ is the increase in the argument of $\det\, A(x^*, \xi^*)$ around $\Gamma$ oriented counterclockwise.
\end{thm}


\section{A nonlocal center manifold theorem}\label{CMT-appendix}

In this section, we record a version due to Truong, Wahl\'{e}n \& Wheeler~\cite{TWW} of the nonlocal center manifold theorem originally developed by Faye \& Scheel~\cite{FS, FS-corrig}.  
This result is the main analytical tool used throughout Section \ref{sect-reduced-r1} and Section \ref{sect-reduced-r2}.

We consider nonlocal nonlinear parameter-dependent problems of the form
\begin{equation} \label{original} \T v + \N(v, \mu) = 0,\end{equation}
where
\[\T v = v + \K*v,\]
in the weighted Sobolev spaces $H^m_{-\eta}$ for some $\eta > 0$ and positive integer $m$. $\T$ is referred to as the linear part and $\N$ as the nonlinear part of~\eqref{original}. Before introducing the modified equation, we define a cutoff operator $\chi$ which is invariant under all translations and reversible symmetries. The translation map by $t \in \R$, that is $\p \mapsto \p(\, \cdot \, + t)$, is denoted by $S_t$. First, let $\cchi\colon \R \to \R$ be a smooth cutoff function satisfying $\cchi=1$ for $|x|<1$, 0 for $|x|>2$ and $\sup_{x \in \R}|\cchi'(x)| \leq 2$. Secondly, let $\theta\colon \R \to \R$ be an even and smooth function with 
\[\sum_{j\in \ZZ} \theta(x-j) = 1, \hspace{0.3cm}\text{supp} \, \theta \subset [-1,1], \hspace{0.3cm} \theta \left( \left[0,\tfrac{1}{2}\right]\right) \subset \left[\tfrac{1}{2}, 1 \right],\]
for all $x\in \R$. Define
\begin{equation}\label{cutoff}
\chi\colon v \mapsto \int_\R \cchi(\|S_y\theta \cdot v\|_{H^m})\theta(x-y)v(x)\dif y \hspace{0.3cm}\text{and}\hspace{0.3cm}\chi^\delta \colon v \mapsto \delta \cdot \chi\left(\frac{v}{\delta}\right), \hspace{0.2cm} \delta > 0.
\end{equation}
It has been shown~\cite{FS-corrig} that $\chi:H^m_{-\eta} \to H^m_{\uu}$ is well-defined, Lipschitz continuous, invariant under all $S_t$ and reversible symmetries on $\R$, and its image is contained in a ball in $H^m_{\uu}$. As a consequence, the scaled cutoff $\chi^\delta$ inherits all these properties except for its image, which will be contained in a ball of radius $\delta$ in $H^m_{\uu}$. The modified equation is 
\begin{equation}\label{modified} \T v + \N^\delta(v, \mu) = 0 \hspace{0.5cm}\text{where} \hspace{0.5cm} \N^\delta(v, \mu) \coloneqq \N(\chi^\delta(v), \mu).\end{equation}

Also, let $\Q\colon H^m_{-\eta} \to H^m_{-\eta}$ be a bounded projection on the nullspace $\Ker \, \T$ of $\T$ with a continuous extension to $H^{m-1}_{-\eta}$, such that $\Q$ commutes with the inclusion map from $H^m_{-\eta}$ to $H^m_{-\eta'}$, for all $0 < \eta' < \eta$. 

\begin{hyp}[The linear part $\T$] \label{hyp-T} $ $
  \begin{itemize}
    \item[(i)] There exists $\eta_0 > 0$ such that $\K \in L^1_{\eta_0}$.
    \item[(ii)] The operator
          \[\T\colon v \mapsto v + \K*v, \hspace{0.5cm}  H^m_{-\eta} \to H^m_{-\eta}\]
          is Fredholm for $\eta \in (0, \eta_0)$, its nullspace $\Ker \, \T$ is finite-dimensional and $\T$ is onto. 
  \end{itemize}
\end{hyp}

\begin{rem}
A straightforward application of Young's inequality shows that Hypothesis \ref{hyp-T}(i) implies the operator $\mathcal{T}\colon H^m_{-\eta}\to H^m_{-\eta}$
is bounded for each $\eta\in(0,\eta_0)$ for each choice of $c_0$. 	 In the works \cite{FS,FS-corrig} the authors additionally assumed
that $K'\in L^1_{\eta_0}$ for some $\eta_0 > 0$ which, as seen from Proposition \ref{K-prop}, does not hold for the current case.  This assumption, however,
is used to guarantee Hypothesis \ref{hyp-T}(ii) which, here, we instead require directly.
\end{rem}

\begin{hyp}[The nonlinear part $\N$] \label{hyp-N} There exist $k \geq 2$, a neighborhood $\U$ of $0 \in H^m_{-\eta}$ and $\V$ of $0 \in \R$, such that for all sufficiently small $\delta > 0$, we have
  \begin{itemize}
    \item[(a)] $\N^\delta \colon H^m_{-\eta} \times \V \to H^m_{-\eta}$ is $\C^k.$ Moreover, for all non-negative pairs $(\zeta, \eta)$ such that $0 < k\zeta <\eta < \eta_0$, $\D_v^l\N^\delta(\, \cdot \,, \mu) \colon (H_{-\zeta}^j)^l \to H_{-\eta}^j$ is bounded for all $0 < l\zeta \leq \eta < \eta_0$ and $0 \leq l \leq k$, and is Lipschitz in $v$ for $1 \leq l \leq k-1$ uniformly in $\mu \in \V$.
        \item[(b)] $\N^\delta$ commutes with translations of $v$,
    \[\N^\delta(S_t v, \mu) = S_t \N(v, \mu), \hspace{0.5cm} \text{for all} \, \, t \in \R.\]
    	\item[(c)] $\N^\delta(0,0) = 0$, ${\D}_v\N^\delta(0,0) = 0$ and as $\delta \to 0$, the Lipschitz constant 
          \[\text{Lip}_{H^m_{-\eta}\times \V} \N^\delta = \O(\delta + |\mu|).\]
  \end{itemize}
\end{hyp}
Let $v\colon \R \to \R$ be a function. A symmetry is a triple $(\rho, S_t, \kappa) \in \mathbf{O}(1) \times (\R \times \mathbf{O}(1))$, acting on $v$ in the following way: the orthogonal linear transformation $\rho \in \mathbf{O}(1)$ acts on the value $v(x)\in \R$, while $S_t$ and $\kappa$ act on the variable $x \in \R$. A symmetry $(\rho, S_t, \kappa)$ is equivariant if $\kappa = \Id$, and reversible otherwise. 
\begin{hyp}[Symmetries] \label{hyp-S} There exists a symmetry group $S$, under which the equation is invariant, that is
  \[\gamma (\T v) = \T (\gamma v), \hspace{0.5cm} \N(\gamma v, \mu) = \gamma \N(v, \mu), \hspace{0.5cm} \text{for all $\gamma \in S$,}\]
such that $S$ contains all translations on the real line. 
\end{hyp}

\begin{thm} \label{abstract-CMT} Assume Hypotheses~\ref{hyp-T},~\ref{hyp-N} and~\ref{hyp-S} are met for the~\eqref{original}. Then, by possibly shrinking the neighborhood $\V$ of $0\in \R$, there exists a cutoff radius $\delta > 0$, a weight $\eta>0$ and a map
  \[\Psi\colon \Ker  \T \times \V \subset H^m_{-\eta}\times \R \to \Ker \Q \subset H^m_{-\eta}\]
with the center manifold
  \[\M_0^\mu =\big\{ v_0 + \Psi(v_0, \mu) \, \big| \, v_0 \in \Ker \T, \mu \in \V \big\} \subset H^m_{-\eta},\]
as its graph for each $\mu$. The following statements hold:
  \begin{itemize}
    \item[(i)] (smoothness) $\Psi \in \C^k$, where $k$ is as in Hypothesis~\ref{hyp-N};
    \item[(ii)] (tangency) $\Psi(0,0) = 0$ and ${{\D_{v_0}}}\Psi(0,0);$
    \item[(iii)] (global reduction) $\M_0^\mu$ consists precisely of functions $v$ such that $v \in H^m_{-\eta}$ is a solution of the modified equation~\eqref{modified} with parameter $\mu$;
    \item[(iv)] (local reduction) any function $v$ solving~\eqref{original} with $\|v\|_{H^m_{\uu}} \lesssim \delta$ is contained in $\M_0^\mu$;
    \item[(v)] (translation invariance) the shift $S_t$, $t \in \R$ acting on $\M_0^\mu$ induces a $\mu$-dependent flow
          \[\Phi_t\colon \Ker \T \to \Ker \T\]
          through $\Phi_t = \Q \circ S_t \circ (\Id + \Psi)$;
    \item[(vi)] (reduced vector field) the reduced flow $\Phi_t(v_0, \mu)$ is of class $\C^k$ in $v_0, \mu, t$ and is generated by a reduced parameter dependent vector field $f$ of class $\C^{k-1}$ on the finite-dimensional $\Ker \T$;
    \item[(vii)] (correspondence) any element $v = v_0 + \Psi(v_0, \mu)$ of $\M_0^\mu$ corresponds one-to-one to a solution of
          \[\frac{\dif v_0}{\dif t} = f(v_0) \coloneqq \frac{\dif}{\dif t} \Q(S_t v)\Big|_{t=0};\]
    \item[(viii)] (equivariance) $\Ker \T$ is invariant under $\Gamma$ and $\Q$ can be chosen to commute with all $\gamma \in S$. Consequently, $\Psi$ commutes with $\gamma \in S$ and $\M_0^\mu$ is invariant under $\Gamma$. Finally, the reduced vector field $f$ in item (vi) commutes with all translations $S_t$ and anticommutes with reversible symmetries in $S$.
  \end{itemize}
\end{thm}

\section{Coefficients in center manifold reduction}

In this appendix, we compute the center manifold coefficients $\psi_{pqlmn} \coloneqq \Psi_{pqlmn}''''(0)$ up to second-order terms in Proposition~\ref{trunc-syst}
as well as the coefficients
\[
\psi_{10001}, \hspace{0.2cm} \psi_{20000}, \hspace{0.2cm}\psi_{00200}, \hspace{0.2cm} \psi_{10010}, \hspace{0.2cm}\psi_{01100}, \hspace{0.2cm} \psi_{10200}, \hspace{0.2cm} \psi_{30000}
\]
from Section~\ref{sect-reduced-r2}. The proof Proposition~\ref{trunc-syst} observes that $\T$, $\Id-\T$ and $\Q$ map even to even, and odd to odd functions. Also, the basis functions of $\Ker \T$ are either even or odd functions. Using these, vanishing $\psi_{rslmn}$ are identified and excluded. Then, linear equations for non-vanishing $\psi_{rslmn}$ are written down. To compute $\psi_{rslmn}$, we will extensively use
\begin{equation}
\label{even-multiplier}
\begin{split}
&\mathbf{m}({\D})\colon x^{2k}\cos(yx) \mapsto \sum_{j=0}^k \binom{2k}{2j} (-1)^{j}\mathbf{m}^{(2j)}(y)\cdot x^{2(k-j)}\cos(yx) \\
&\hspace{2cm}+  \sum_{j=0}^{k-1} \binom{2k}{2j+1}(-1)^{j}\mathbf{m}^{(2j+1)}(y)\cdot x^{2(k-j)-1}\sin(yx)\\
&\mathbf{m}({\D}) \colon x^{2k+1}\sin(yx) \mapsto \sum_{j=0}^k \binom{2k+1}{2j} (-1)^{j}\mathbf{m}^{(2j)}(y)\cdot x^{2(k-j)+1}\sin(yx) \\
&\hspace{2cm}+  \sum_{j=0}^{k} \binom{2k+1}{2j+1}(-1)^{j+1}\mathbf{m}^{(2j+1)}(y)\cdot x^{2(k-j)}\cos(yx),\\
\end{split}
\end{equation}
where $\mathbf{m}\colon \R \to \R$ is an even multiplier and $y \in \R$. Finally, we observe that if $f$ satisfies $\T f = g$, then $h\coloneqq f - \Q f$ satisfies $\T h = g$ and $\Q h = 0$.

\subsection{For generalized solitary waves}
\label{center-mani-coeff-r1}
Here, let $\tau,c_0$ satisfy \eqref{parameter-r1} and note, specifically, that $c_0=1$ here.
Equation~\eqref{eq} with $\p = \Q_1\p + \Psi$ is
\[
\T \Psi  - \mu \Q_1 \p+ (\Id - \T)(\Q_1\p + \Psi)^2 = 0.
\]
By noting that second-order $\mu$-inhomogeneous terms come from $\T\Psi$ and $(\Id - \T) (\Q_1 \p)^2$, the linear equations from grouping
 $A^2, B^2, C^2, D^2$, $AB, AC, AD$, $BC, BD$ and $CD$ terms are given by
\begin{center}
\def\arraystretch{1.5}
\begin{tabular}{l l}
	$ \T \Psi_{20000} + (\Id - \T) 1 = 0,$ & $ \T \Psi_{00200} + (\Id - \T) \cos^2(k_{0} x)= 0,$ \\
	$\T \Psi_{02000} + (\Id - \T)\, x^2 = 0, $ & $ \T \Psi_{00020} + (\Id - \T) \sin^2(k_{0} x)= 0, $ \\
	$ \T \Psi_{11000} + 2(\Id - \T)\,x = 0,$ & $ \T \Psi_{10100} + 2(\Id - \T)\cos(k_{0}x) = 0, $ \\
	$ \T \Psi_{10010} + 2(\Id - \T)\sin(k_{0}x) =0,\hspace{1cm}$ & $\T \Psi_{01100} + 2(\Id - \T)\, x \cos(k_{0}x)  = 0, $ \\
	$\T \Psi_{00110} + (\Id - \T)\sin(2k_{0} x)= 0, $ & $\T \Psi_{01010} + 2(\Id - \T) \, x \sin(k_0x) = 0,$
\end{tabular}
\end{center} 
and, by noting that the $\mu$-homogeneous terms come from $\T\Psi$ and $-\mu \Q_1 \p$,
\begin{center}
	\def\arraystretch{1.5}
	\begin{tabular}{l l}
		$\T \Psi_{10001} - 1 = 0, \hspace{0.5cm} $ & $\T \Psi_{00101} - \cos(k_{0}x) = 0, $ \\
		$\T \Psi_{01001} - x = 0, $ & $\T \Psi_{00011} - \sin(k_{0}x) = 0. $ 
	\end{tabular}
\end{center}
Note that equations arising from grouping $AB, AD, BC, CD, \mu B$ and $\mu D$ terms are excluded here since they involve only odd functions. 
Using~\eqref{even-multiplier} with $\mathbf{m} = \ell$ and $y =0, k_0$ or $2k_0$, we arrive at 
\begin{center}
	\def\arraystretch{1.5}
	\begin{tabular}{l l}
		$\T \Psi_{10001} = 1, $ & $ \T \Psi_{00101} = \cos(k_{0}x),$ \\
		$\T \Psi_{20000} = -1, $ & $ \T \Psi_{10100} =- 2\cos(k_{0}x), $ \\ 
		$\T \Psi_{02000} = - x^2 +\ell''(0)$ & $ \T \Psi_{00200} =  -\frac{1}{2}-\frac{1}{2}\ell(2k_{0})\cos(2k_{0}x),$ \\
		$\T \Psi_{01010} = -2x \sin(k_0x) + 2\ell'(k_{0}) \cos(k_{0} x),\hspace{0.5cm} $ & $ \T \Psi_{00020} =  -\frac{1}{2}+ \frac{1}{2}\ell(2k_{0})\sin(2k_{0}x),$ \\
	\end{tabular}
\end{center}
all subjected to the condition $\Q_1 \Psi_{pqlmn} = 0$, which ensures uniqueness.

Let $\sigma = \ell''(0)^{-1} = (1/3-\tau)^{-1}$. Lengthy but straightforward calculations employing~\eqref{even-multiplier} with $\mathbf{m}=1-\ell$ now yield
\begin{center}
	\def\arraystretch{2.2}
\begin{tabular}{l l}
$\psi_{10001} = -\psi_{20000} = 2\sigma k_0^2,$ & $\psi_{00101} = -\frac{1}{2}\psi_{10100}= -\dfrac{2k_0^3}{\ell'(k_0)},$ \\
$\psi_{02000} =  - \dfrac{\ell''''(0) - 6\sigma^{-2}}{3}\sigma^2 k_0^2 - 4\sigma,$ & $\psi_{01010} = 2\dfrac{\ell''(k_0) - 2\ell'(k_0)^2}{\ell'(k_0)^2} k_0^3 - \dfrac{10}{\ell'(k_0)}k_0^2,$\\
$\psi_{00200} = \dfrac{8\ell(2k_0)}{\ell(2k_0)-1} k_0^4 - \sigma k_0^2,$ & $\psi_{00020} = -\dfrac{8\ell(2k_0)}{\ell(2k_0)-1}k_0^4 - \sigma k_0^2.$
\end{tabular}
\end{center}

\subsection{For modulated solitary waves}
\label{center-mani-coeff-r2}

Now, let $\tau,c_0$ satisfy \eqref{parameter-r2} and note in this case that 
both $c_0$ and $\tau_0$ are parametrized by $s\in(0, \infty)$. The linear equations for the center manifold coefficients are
\[\begin{split}&\T \Psi_{10001} = \frac{1}{c_0} (\Id - \T)\cos(sx) ,\\
&\T \Psi_{20000} = - \frac{1}{c_0}(\Id - \T) \cos^2(sx), \\
&\T \Psi_{10100} = - \frac{1}{c_0}(\Id-\T)\sin(2sx),\\
&\T \Psi_{00200} = - \frac{1}{c_0}(\Id-\T)\sin^2(sx), \\
&\T \Psi_{10010} = \T \Psi_{01100} = -\frac{1}{c_0}(\Id - \T)x\sin(2sx), \\
&\T \Psi_{30000} = -\frac{2}{c_0}(\Id-\T)\cos(sx)\Psi_{20000}, \\
&\T \Psi_{10200} = -\frac{2}{c_0}(\Id - \T) (\cos(sx)\Psi_{00200} + \sin(sx)\Psi_{10100}),
\end{split}\]
where all $\Psi_{pqlmn}$ are subject to $\Q_2\Psi_{pqlmn}= 0$ and $\Q_2$ is given in~\eqref{proj-r2}. Using~\eqref{even-multiplier} with $\mathbf{m} = 1-c_0\ell, y = 0, s, 2s$ or $3s$ again, then evaluating $\Psi_{pqlmn}''''(0) = \psi_{pqlmn}$ gives
\[\begin{split}
& \psi_{10001} = -8s^2e, \\
& \psi_{20000} = s^4(a+9b), \\
& \psi_{00200} = s^4(a-9b), \\
& \psi_{10010} = \psi_{01100} = 9s^4c-48s^3b, \\
&\psi_{30000} = \Big(\big(-2a(a+b)-18b(a+b)+128bd - \frac{9s}{2}(a-3b)c\big)s^2 \\
&\hspace{2.5cm}+24 s^2b(a-3b)+8(2a+b)e\Big)s^2, \\
& \psi_{10200} = \Big(\big(-2a(a-b)-18b(a-b) - 128 bd - \frac{9s}{2}(a+3b)c\big)s^2\\
&\hspace{2.5cm}+24 s^2 b(a+3b) + 8(2a-b)e\Big)s^2\\
&\hspace{2.5cm}+\Big(\big(-54sbc + 4ab - 36b^2 - 256 bs\big)s^2 + 288 s^2 b^2 + 16 be \Big)s^2,
\end{split}\]
with 
\[\begin{aligned}&a=\frac{1}{2c_0}\left(1-\frac{1}{1-c_0}\right),& &b = \frac{1}{2c_0}\left(1-\frac{1}{1-c_0\ell(2s)}\right),&\\
&c=\frac{\ell'(2s)}{(1-c_0\ell(2s))^2},& &d=\frac{1}{2c_0}\left(1-\frac{1}{1-c_0\ell(3s)}\right),& \\
\end{aligned}
\]
and finally
\[e = \frac{1}{c_0^2\ell''(s)}.\]

\section{Coefficients in normal form reduction}
\subsection{For generalized solitary waves}
\label{appendix-normal-form-r1}
Our goal here is to compute coefficients $p_0,p_1,p_2,p_3,q_0$ and $q_1$ in Section~\ref{normal-form-r1}. The expansion of $\Phi$ in $(\Aa, \Bb, \cC, \Cc)$ and $\mu$ up to second-order terms is
\[\begin{split}&\Phi(\Aa, \Bb, \cC, \Cc, \mu) \\
&= \phi_{00001} \mu + \phi_{10001}\mu \Aa + \phi_{01001}\mu \Bb + \phi_{00101} \mu \cC + \phi_{00011} \mu \Cc  + \phi_{20000} \Aa^2 + \phi_{11000} \Aa \Bb \\
&\hspace{0.5cm}+ \phi_{10100} \Aa \cC + \phi_{10010} \Aa \Cc + \phi_{02000} \Bb^2 + \phi_{01100} \Bb\cC + \phi_{01010} \Bb \Cc + \phi_{00200} \cC^2 \\
&\hspace{0.5cm}+ \phi_{00110} |\cC|^2 + \phi_{00020} \Cc^2 + \O(|\mu^2|+(|\mu|+|(\Aa, \Bb, \cC, \Cc)|)^3)).
\end{split}\]
Denote the coefficients in front of $A^pB^qC^lD^m\mu^n$ in~\eqref{reduced-syst} by $\psi_{pqlmn}$. We also Taylor expand the nonlinear term 
\[\Rr(U, \mu) = \mu \Rr_{11}(U) + \Rr_{20}(U, U) + \O(|\mu|^2 + |(U, \mu)|^3),\]
where 
\[\begin{split} \Rr_{11}(x,y,z,w) = \frac{1}{k_0^3}\begin{pmatrix}
0 \\ k_0\psi_{10001}x + k_0 \psi_{00101}z \\ 0 \\-\psi_{10001}x - \psi_{00101}z
\end{pmatrix} \hspace{0.3cm}\text{and}\hspace{0.3cm}\Rr_{20}(U, \tilde{U}) = \frac{1}{k_0^3} \begin{pmatrix}
0 \\ k_0 H(U, \tilde{U}) \\ 0 \\ -H(U, \tilde{U})
\end{pmatrix}
\end{split}\]
with 
\[\begin{split}
&H(U, \tilde{U}) = H((x,y,z,w), (\tilde{x}, \tilde{y}, \tilde{z}, \tilde{w})) \\
& = \psi_{20000}x \tilde{x} + \frac{\psi_{10100}}{2}(x\tilde{z}+z\tilde{x}) + \psi_{02000} y \tilde{y} +\frac{\psi_{01010}}{2}(y \tilde{w}+ w\tilde{y}) + \psi_{00200}z\tilde{z} + \psi_{00020}w\tilde{w}.
\end{split}\]
Plugging $U = \Aa \xi_0 + \Bb \xi_1 + \cC \zeta + \Cc \overline{\zeta} + \Phi$ into~\eqref{diff-eq-r1}, relevant linear equations are identified
\[\begin{split}
\O(\mu)\colon \hspace{0.5cm} &p_0\xi_1 = \Ll\phi_0  \\
\O(\mu \Aa) \colon \hspace{0.5cm} &p_1\xi_1 + p_0\phi_{11000} = \Ll\phi_{10001} + \Rr_{11}(\xi_0)\\
\O(\Aa^2) \colon \hspace{0.5cm} & p_2\xi_1 = \Ll\phi_{20000} + \Rr_{20}(\xi_0, \xi_0)\\
\O(|\cC|^2) \colon \hspace{0.5cm} & p_3\xi_1 = \Ll\phi_{00110}+2\Rr_{20}(\zeta,\overline{\zeta})\\
\O(\mu \cC)\colon \hspace{0.5cm} & \i q_0 \zeta + p_0\phi_{01100} + \i k_0 \phi_{00101} = \Ll\phi_{00101}+\Rr_{11}(\zeta)\\
\O(\Aa\cC)\colon \hspace{0.5cm} & \i q_1 \zeta + \i k_0 \phi_{10100} = \Ll\phi_{10100}+2\Rr_{20}(\xi_0, \zeta).
\end{split}\]

Since $\xi_1$ is not in the range of $\Ll$, $p_0 = 0$. Similarly, the equation from $\O(\mu \Aa)$ terms 
\[\Ll\phi_{10001} = p_1\xi_1-\Rr_{11}(\xi_0) = \frac{1}{k_0^3}\begin{pmatrix}0 \\ k_0^3 p_1 - k_0 \psi_{10001} \\ 0 \\ \psi_{10001} \end{pmatrix}\]
is solvable if and only if $k_0^3 p_1 - k_0 \psi_{10001} = 0$. Equations for $p_2$  and $p_3$  are handled in the same fashion. To solve for $q_0$, we note that
\[(\Ll-\i k_0)\phi_{00101} = \i q_0 \zeta - \Rr_{11}(\zeta)\]
Writing $\phi_{00101} = x_0 \xi_0 + x_1 \xi_1 + x_3 \zeta + \overline{x_3 \zeta}$, equation~\eqref{L-R-basis} can be used to show that $x_3\zeta = 0$ on the left-hand side. The right-hand side in the basis $\{\xi_0, \xi_1, \zeta, \overline{\zeta}\}$ is
\[\i q_0 \zeta - \Rr_{11}(\zeta) = -\frac{1}{k_0^2}\psi_{00101}\xi_1 + \left(\i q_0 - \frac{\i}{2k_0^3}\psi_{00101}\right)\zeta + \frac{\i}{2k_0^3}\overline{\zeta},\]
which gives $q_0 = (2k_0^3)^{-1} \psi_{00101}$. Using results from Appendix~\ref{center-mani-coeff-r1} and writing $\sigma = (1/3-\tau)^{-1}$, we get
\[\begin{split}
&p_0 = 0, \hspace{0.2cm} p_1 = 2\sigma = -p_2, \hspace{0.2cm} p_3 =  -4\sigma, \hspace{0.2cm} q_0 = -\frac{1}{\ell'(k_0)} \hspace{0.2cm} \text{and} \hspace{0.2cm}q_1 = \frac{2}{\ell'(k_0)}.
\end{split}\]

\subsection{For modulated solitary waves}
\label{appendix-normal-form-r2}
As before, the Taylor expansion of $\Phi(\Aa, \Bb, \overline{\Aa}, \overline{\Bb}, \mu)$ is
\[\Phi(\Aa, \Bb, \overline{\Aa}, \overline{\Bb}, \mu) = \sum_{\substack{2\leq p+q+l+m+n \leq 3 \\ n \ge 1}} \phi_{pqlmn} \Aa^p \Bb^q \overline{\Aa}^l \overline{\Bb}^m \mu^n + \cdots.\]
Our goal here is to  determine the Taylor expansion of $\Rr(U, \mu)$ from \eqref{diff-eq-r2} up to order three.  Using the symmetries as in the proof of Proposition~\ref{trunc-syst}, contributing terms are 
\[\begin{split}
&\mu A, \mu D, \mu^2A, \mu^2 D,\\
& A^2, B^2, C^2, D^2, AD, BC, \mu A^2, \mu B^2, \mu C^2, \mu D^2, \mu AD, \mu BC, \\
& A^3, D^3, A^2D, AB^2, AC^2, AD^2, ABC, B^2D, BCD, C^2D.
\end{split}\]
In short, if a multiplication between a pair or a triple of $\cos(sx), x\cos(sx), \sin(sx), x\sin(sx)$ is an even function in $x$, the product of their coefficients $A,B, C, D$ will contribute to $\Psi(A, B, C, D,\mu)''''(0)$. Denote the coefficients of $A^pB^qC^lD^m\mu^n$ by $\psi_{pqlmn}$. The Taylor expansion of $\Rr(U, \mu)$ is 
\[\begin{split}\Rr(U, \mu) &= \Rr_{00} + \Rr_{01}\mu + \Rr_{10}U + \Rr_{11}\mu U + \Rr_{20}(U, U) \\
&\qquad + \mu \Rr_{21}(U, U) + \mu^2 \Rr_{12}U + \Rr_{30}(U, U, U),
\end{split}\]
where relevant terms for us are
\[\begin{aligned}\color{black}&\Rr_{01}=\frac{1}{2s^3} \begin{pmatrix} 0 \\ -sH_{0}\\ H_0 \\ 0 \end{pmatrix},& &\Rr_{11}(U) = \frac{1}{2s^3}\begin{pmatrix}0 \\ -sH_{1}(U) \\H_1(U)\\ 0 \end{pmatrix}, \\
&\Rr_{20}(U, \tilde{U}) = \frac{1}{2s^3}\begin{pmatrix}0\\ -sH_2(U, \tilde{U}) \\ H_2(U, \tilde{U}) \\0 \end{pmatrix},&  &\Rr_{30}(U, \tilde{U}, \hat{U}) = \frac{1}{2s^3} \begin{pmatrix}
0 \\-s H_3(U, \tilde{U}, \hat{U}) \\ H_3(U, \tilde{U}, \hat{U}) \\ 0
\end{pmatrix},
\end{aligned}\]
with $U=(x,y,z,w), \tilde{U}= (\tilde{x}, \tilde{y}, \tilde{z}, \tilde{w}), \hat{U} = (\hat{x}, \hat{y}, \hat{z}, \hat{w})$, and
\begin{align*}
&H_0 = 0,\\
&H_1(x,y,z,w) \\
&\quad= \psi_{10001}x + \psi_{00011}w, \\
\hspace{0.5cm} &H_2((x,y,z,w), (\tilde{x}, \tilde{y}, \tilde{z}, \tilde{w})) \\
&\quad= \psi_{20000}x\tilde{x} + \psi_{02000} y \tilde{y} + \psi_{00200} z \tilde{z} + \psi_{00020}w\tilde{w} + \frac{\psi_{10010}}{2}(x\tilde{w}+\tilde{x}w) + \frac{\psi_{01100}}{2}(y \tilde{z} + \tilde{y}z),\\
&H_3((x,y,z,w), (\tilde{x},\tilde{y}, \tilde{z},\tilde{w}), (\hat{x}, \hat{y}, \hat{z}, \hat{w})) \\
&\quad= \psi_{30000} x \tilde{x} \hat{x} +\psi_{00030} w \tilde{w} \hat{w}\\
&\hspace{0.65cm} + \frac{\psi_{20010}}{3}(x \tilde{x}\hat{w}+x\hat{x}\tilde{w} + \tilde{x}\hat{x}w)+ \frac{\psi_{12000}}{3}(x\tilde{y}\hat{y}+\tilde{x}y\hat{y}+\hat{x}y\tilde{y})\\
&\hspace{0.65cm}+\frac{\psi_{10200}}{3}(x\tilde{z}\hat{z}+\tilde{x}z\hat{z}+\hat{x}z\tilde{z})+\frac{\psi_{10020}}{3}(x\tilde{w}\hat{w}+\tilde{x}w\hat{w}+\hat{x}w\tilde{w})  \\
&\hspace{0.65cm} + \frac{\psi_{11100}}{6}(x\tilde{y}\hat{z}+x\hat{y}\tilde{z}+\tilde{x}y\hat{z}+\tilde{x}\hat{y}z+\hat{x}y\tilde{z}+\hat{x}\tilde{y}z)\\
&\hspace{0.65cm} + \frac{\psi_{01110}}{6}(y\tilde{z}\hat{w}+y\hat{z}\tilde{w}+\tilde{y}z\hat{w}+\tilde{y}\hat{z}w+\hat{y}z\tilde{w}+\hat{y}\tilde{z}w)\\
&\hspace{0.65cm}  +\frac{\psi_{02010}}{3}(y\tilde{y}\hat{w}+y\hat{y}\tilde{w}+\tilde{y}\hat{y}w) +\frac{\psi_{00210}}{3}(z\tilde{z}\hat{w}+z\hat{z}\tilde{w}+\tilde{z}\hat{z}w).
\end{align*}
Let $\zeta^*_1 = \frac{1}{2}(0,1,0,-\i)^{\text{T}}$. It is a vector orthogonal to the range of $\Ll-\i s$ and satisfies
\[\langle\zeta_0, \zeta_1^*\rangle = 0, \hspace{0.3cm} \langle \overline{\zeta_0}, \zeta_1^*\rangle = 0, \hspace{0.3cm} \langle\zeta_1, \zeta_1^* \rangle = 0, \hspace{0.3cm}\langle\overline{\zeta_1}, \zeta_1^*\rangle = 0,\]
and $R\zeta_1^* = -\overline{\zeta^*_1}.$ Equations (D.45) and (D.47),~\cite[Appendix D.2]{HI}, give
\[\begin{split}&q_0 = \langle\Rr_{11} \zeta_0 + 2\Rr_{20}(\zeta_0, \phi_{00001}), \zeta_1^*\rangle,\\
&q_1 = \langle 2 \Rr_{20}(\zeta_0, \phi_{10100}) + 2 \Rr_{20}(\overline{\zeta_0}, \phi_{20000})+ 3\Rr_{30}(\zeta_0, \zeta_0, \overline{\zeta_0}), \zeta_1^* \rangle,
\end{split}\]
respectively. Here, $\phi_{00001}, \phi_{10100}, \phi_{20000}$ satisfy
\[\begin{split}
&\Ll\phi_{00001} + \Rr_{01} = 0, \hspace{0.2cm}\Ll\phi_{10100}+2\Rr_{20}(\zeta_0, \overline{\zeta_0}) = 0, \hspace{0.2cm}(\Ll-2\i s)\phi_{20000} + \Rr_{20}(\zeta_0, \zeta_0) = 0.
\end{split}\]
A computation gives
\[\phi_{00001} = 0, \hspace{0.3cm} \phi_{10100} = \frac{\psi_{20000}+\psi_{02000}}{s^3}\begin{pmatrix} 2/s \\0 \\ 0 \\ 1 \end{pmatrix}, \hspace{0.3cm} \phi_{20000} =  \frac{\psi_{20000}-\psi_{02000}}{9s^4}\begin{pmatrix}1\\ 3\i s \\ -\i \\ -3s/2\end{pmatrix},\]
which in turn yields
\[\begin{split}
&q_0 = -\frac{1}{4s^2}\psi_{10001}, \\
&q_1 = -\frac{\psi_{20000}+\psi_{00200}}{2s^5}\left(\frac{2}{s}\psi_{20000}+\frac{\psi_{10010}}{2}\right)\\
&\hspace{1.25cm} -\frac{1}{18s^6}(\psi_{20000}-\psi_{00200}) \left(\psi_{20000}-\psi_{00200}+\frac{3s}{2}\psi_{01100}-\frac{3s}{4}\psi_{01100}\right)\\
&\hspace{1.25cm}-\frac{3}{4s^2}\left(\psi_{30000}+\frac{\psi_{10200}}{3}\right). 
\end{split}\]

\bibliography{extreme_traveling}

\begin{bibdiv}
\begin{biblist}

\bib{AK2}{incollection}{
      author={Amick, C.~J.},
      author={Kirchg\"{a}ssner, K.},
       title={Solitary water-waves in the presence of surface tension},
        date={1987},
   booktitle={Dynamical problems in continuum physics ({M}inneapolis, {M}inn.,
  1985)},
      series={IMA Vol. Math. Appl.},
      volume={4},
   publisher={Springer, New York},
       pages={1\ndash 22},
         url={https://doi-org.www2.lib.ku.edu/10.1007/978-1-4612-1058-0_1},
}

\bib{AK}{article}{
      author={Amick, J.~Charles},
      author={Kirchgässner, Klaus},
       title={A theory of solitary water-waves in the presence of surface
  tension},
        date={1989},
     journal={Arch. Rational Mech. Anal.},
      volume={105},
       pages={1\ndash 49},
}

\bib{A16}{article}{
      author={Arnesen, Mathias~Nikolai},
       title={Existence of solitary-wave solutions to nonlocal equations},
        date={2016},
        ISSN={1078-0947},
     journal={Discrete Contin. Dyn. Syst.},
      volume={36},
       pages={3483\ndash 3510},
}

\bib{BS}{article}{
      author={Bakker, Bente},
      author={Scheel, Arnd},
       title={Spatial {H}amiltonian identities for nonlocally coupled systems},
        date={2018},
        ISSN={2050-5094},
     journal={Forum Math. Sigma},
      volume={6},
       pages={e22},
}

\bib{Baele1}{article}{
      author={Beale, J.~Thomas},
       title={Exact solitary water waves with capillary ripples at infinity},
        date={1991},
        ISSN={0010-3640},
     journal={Comm. Pure Appl. Math.},
      volume={44},
       pages={211\ndash 257},
         url={https://doi-org.www2.lib.ku.edu/10.1002/cpa.3160440204},
}

\bib{BEP}{article}{
      author={Bruell, Gabriele},
      author={Ehrnstr\"{o}m, Mats},
      author={Pei, Long},
       title={Symmetry and decay of traveling wave solutions to the {W}hitham
  equation},
        date={2017},
        ISSN={0022-0396},
     journal={J. Differential Equations},
      volume={262},
       pages={4232\ndash 4254},
}

\bib{BG}{article}{
      author={Buffoni, Boris},
      author={Groves, M.~D.},
       title={A multiplicity result for solitary-capillary waves in deep water
  via critical point theory},
        date={1999},
     journal={Arch. Rational Mech. Anal},
      volume={146},
       pages={183\ndash 220},
}

\bib{BGT}{article}{
      author={Buffoni, Boris},
      author={Groves, M.~D.},
      author={Toland, J.~F.},
       title={A plethora of solitary gravity-capillary water waves with nearly
  critical bond and froude numbers},
        date={1996},
     journal={Philos. Trans. R. Soc. London. Ser. A. Math. Phys. Sci. Eng.},
      volume={354},
       pages={575\ndash 607},
}

\bib{CWW2}{article}{
      author={Chen, Robin~Ming},
      author={Walsh, Samuel},
      author={Wheeler, Miles~H.},
       title={Center manifolds without a phase space for quasilinear problems
  in elasticity, biology, and hydrodynamics},
        date={2019},
        note={Available at \url{https://arxiv.org/abs/1907.04370}},
}

\bib{DI}{article}{
      author={Dias, F.},
      author={Iooss, G.},
       title={Capillary-gravity solitary waves with damped oscillations},
        date={1993},
     journal={Phys. D},
      volume={65},
       pages={399\ndash 423},
}

\bib{EGW}{article}{
      author={Ehrnstr\"{o}m, Mats},
      author={Groves, Mark~D.},
      author={Wahl\'{e}n, Erik},
       title={On the existence and stability of solitary-wave solutions to a
  class of evolution equations of {W}hitham type},
        date={2012},
        ISSN={0951-7715},
     journal={Nonlinearity},
      volume={25},
       pages={2903\ndash 2936},
}

\bib{EJC}{article}{
      author={Ehrnstr\"{o}m, Mats},
      author={Johnson, Mathew~A.},
      author={Claassen, Kyle~M.},
       title={Existence of a highest wave in a fully dispersive two-way shallow
  water model},
        date={2019},
        ISSN={0003-9527},
     journal={Arch. Ration. Mech. Anal.},
      volume={231},
       pages={1635\ndash 1673},
         url={https://doi.org/10.1007/s00205-018-1306-5},
}

\bib{EJMR}{article}{
      author={Ehrnstr\"{o}m, Mats},
      author={Johnson, Mathew~A.},
      author={Maehlen, Ola I.~H.},
      author={Remonato, Filippo},
       title={On the bifurcation diagram of the capillary-gravity {W}hitham
  equation},
        date={2019},
     journal={Springer Nature},
      volume={1},
       pages={217\ndash 313},
}

\bib{EK}{article}{
      author={Ehrnstr\"{o}m, Mats},
      author={Kalisch, Henrik},
       title={Traveling waves for the {W}hitham equation},
        date={2009},
     journal={Differential Integral Equations},
      volume={22},
       pages={1193\ndash 1210},
}

\bib{EW}{article}{
      author={Ehrnstr\"{o}m, Mats},
      author={Wahl\'{e}n, Erik},
       title={On {W}hitham's conjecture of a highest cusped wave for a nonlocal
  dispersive equation},
        date={2019},
        ISSN={0294-1449},
     journal={Ann. Inst. H. Poincar\'{e} Anal. Non Lin\'{e}aire},
      volume={36},
       pages={1603\ndash 1637},
}

\bib{FS}{article}{
      author={Faye, Gr\'{e}gory},
      author={Scheel, Arnd},
       title={Center manifolds without a phase space},
        date={2018},
        ISSN={0002-9947},
     journal={Trans. Amer. Math. Soc.},
      volume={370},
       pages={5843\ndash 5885},
}

\bib{FS-corrig}{article}{
      author={Faye, Gr\'{e}gory},
      author={Scheel, Arnd},
       title={Corrigendum to center manifolds without a phase space},
        date={2020},
        note={Available at \url{https://arxiv.org/pdf/2007.14260.pdf}},
}

\bib{Grushin}{article}{
      author={Gru\v{s}in, V.~V.},
       title={Pseudodifferential operators in {$R^{n}$} with bounded symbols},
        date={1970},
        ISSN={0374-1990},
     journal={Funkcional. Anal. i Prilo\v{z}en},
      volume={4},
       pages={37\ndash 50},
}

\bib{HI}{book}{
      author={Haragus, Mariana},
      author={Iooss, Gérard},
       title={Local bifurcations, center manifolds, and normal forms in
  infinite-dimensional dynamical systems},
   publisher={Springer, London},
        date={2011},
}

\bib{Hur}{article}{
      author={Hur, Vera~Mikyoung},
       title={Wave breaking in the {W}hitham equation},
        date={2017},
        ISSN={0001-8708},
     journal={Adv. Math.},
      volume={317},
       pages={410\ndash 437},
}

\bib{HJ2}{article}{
      author={Hur, Vera~Mikyoung},
      author={Johnson, Mathew~A.},
       title={Modulational instability in the {W}hitham equation for water
  waves},
        date={2015},
        ISSN={0022-2526},
     journal={Stud. Appl. Math.},
      volume={134},
       pages={120\ndash 143},
         url={https://doi-org.www2.lib.ku.edu/10.1111/sapm.12061},
}

\bib{HJ}{article}{
      author={Hur, Vera~Mikyoung},
      author={Johnson, Mathew~A.},
       title={Modulational instability in the {W}hitham equation with surface
  tension and vorticity},
        date={2015},
     journal={Nonlinear Anal.},
      volume={129},
       pages={104\ndash 118},
}

\bib{IP}{article}{
      author={Iooss, G.},
      author={P\'{e}rouème, M.~C.},
       title={Perturbed homoclinic solutions in reversible 1:1 resonance vector
  fields},
        date={1993},
     journal={J. Diff. Equ.},
      volume={102},
       pages={62\ndash 88},
}

\bib{IK1990}{article}{
      author={Iooss, Gérard},
      author={Kirchgässner, Klaus},
       title={Bifurcations d'ondes solitaires en présense d'une faible tension
  superficielle},
        date={1990},
     journal={C.R. Acad. Sci. Paris},
      volume={1},
       pages={265\ndash 268},
}

\bib{IK}{article}{
      author={Iooss, Gérard},
      author={Kirchgässner, Klaus},
       title={Water waves for small surface tension: an approach via normal
  form},
        date={1991},
     journal={Proc. Roy. Soc. Edinburg},
      volume={122},
       pages={267\ndash 299},
}

\bib{JW}{article}{
      author={Johnson, Mathew~A.},
      author={Wright, J.~Douglas},
       title={Generalized solitary waves in the gravity-capillary {W}hitham
  equation},
        date={2020},
        ISSN={0022-2526},
     journal={Stud. Appl. Math.},
      volume={144},
       pages={102\ndash 130},
         url={https://doi.org/10.1111/sapm.12288},
}

\bib{RS_JohnsonBook}{book}{
      author={Johnson, R.~S.},
       title={A modern introduction to the mathematical theory of water waves},
      series={Cambridge Texts in Applied Mathematics},
   publisher={Cambridge University Press, Cambridge},
        date={1997},
        ISBN={0-521-59832-X},
         url={https://doi-org.www2.lib.ku.edu/10.1017/CBO9780511624056},
}

\bib{Kirchgassner}{incollection}{
      author={Kirchg\"{a}ssner, Klaus},
       title={Nonlinearly resonant surface waves and homoclinic bifurcation},
        date={1988},
   booktitle={Advances in applied mechanics, vol. 26},
      series={Adv. Appl. Mech.},
      volume={26},
   publisher={Academic Press, Boston, MA},
       pages={135\ndash 181},
         url={https://doi.org/10.1016/S0065-2156(08)70288-6},
}

\bib{Lombardi1}{article}{
      author={Lombardi, Eric},
       title={Orbits homoclinic to exponentially small periodic orbits for a
  class of reversible systems. {A}pplication to water waves},
        date={1997},
     journal={Arch. Rational Mech. Anal.},
      volume={137},
       pages={227\ndash 304},
}

\bib{Lombardi2}{book}{
      author={Lombardi, Eric},
       title={Oscillatory integrals and phenomena beyond all algebraic orders.
  {W}ith applications to homoclinic orbits in reversible system},
      series={Lecture Notes in Mathematics},
   publisher={Springer-Verlag, Berlin},
        date={2002},
      volume={1741},
}

\bib{Sun}{article}{
      author={Sun, S.~M.},
       title={Existence of a generalized solitary wave solution for water with
  positive {B}ond number less than $1/3$},
        date={1991},
     journal={J. Math. Anal. Appl.},
      volume={156},
       pages={471\ndash 504},
}

\bib{SS}{article}{
      author={Sun, S.~M.},
      author={Shen, M.~C.},
       title={Exponentially small estimate for the amplitude of capillary
  ripples of a generalized solitary waves},
        date={1993},
     journal={J. Math. Anal. Appl.},
      volume={172},
       pages={533\ndash 566},
}

\bib{TWW}{article}{
      author={Truong, Tien},
      author={Wahl\'{e}n, Erik},
      author={Wheeler, Miles~H.},
       title={Global bifurcation of solitary waves for the {W}hitham equation},
        date={2021},
     journal={Math. Ann.},
        note={\url{https://doi.org/10.1007/s00208-021-02243-1}},
}

\bib{Whitham}{incollection}{
      author={Whitham, G.~B.},
       title={Variational methods and applications to water waves},
        date={1970},
   booktitle={Hyperbolic equations and waves ({R}encontres, {B}attelle {R}es.
  {I}nst., {S}eattle, {W}ash., 1968)},
       pages={153\ndash 172},
}

\bib{Whitham_book}{book}{
      author={Whitham, G.~B.},
       title={Linear and nonlinear waves},
   publisher={John Wiley \& Sons},
        date={1999},
}

\end{biblist}
\end{bibdiv}

\end{document}